\title{Picard-Fuchs system for family of Kummer surfaces as subsystem of GKZ hypergeometric system}
\author{Atsuhira Nagano}
\documentclass[10pt]{article}
\usepackage{mathrsfs,bm,graphics,amsfonts,amsthm,amssymb,amsmath,amscd,booktabs}
\usepackage[pdftex]{graphicx}
\def\bigzerou{\smash{\lower1.7ex\hbox{\b 0}}}

\newtheorem{thm}{Theorem}[section]

\newtheorem{lem}{Lemma}[section]
\newtheorem{prop}{Proposition}[section]
\newtheorem{rem}{Remark}[section]

\newtheorem{cor}{Corollary}[section]

\def\comment#1{{ }}
\makeatletter
  
      \@addtoreset{equation}{section}
 \allowdisplaybreaks

\begin{document}
\maketitle
\setlength{\baselineskip}{13 pt}
 \renewcommand{\thefootnote}{\fnsymbol{footnote}}

\begin{abstract}
We determine a simple expression of the Picard-Fuchs system for a family of all Kummer surfaces for principally polarized Abelian surfaces.
It is given by a system of linear partial differential equations in three variables of rank five.
Our results are based on a Jacobian elliptic fibration on Kummer surfaces and a GKZ hypergeometric system suited to the elliptic fibration.
\end{abstract}

\footnote[0]{Keywords:  Kummer surfaces ; Picard-Fuchs systems  ;  Hypergeometic functions.  }
\footnote[0]{Mathematics Subject Classification 2020:  Primary 14J28 ; Secondary  33C75, 14J27.}

\section*{Introduction}

A Kummer surface ${\rm Kum} (A)$ for an Abelian surface $A$ is a significant algebraic $K3$ surface.
It is classically known that ${\rm Kum} (A)$  is given by a quartic surface in the projective space $\mathbb{P}^3(\mathbb{C})$ with $16$ nodes.
Since the period points of marked Kummer surfaces are dense in the period domain of $K3$ surfaces,
they play an essential role in the proof of the Torelli theorem for $K3$ surfaces, which guarantees the injectivity of the period mapping,  by  Piateckii-Shapiro and  Shafarevich \cite{PS}. 
Also, Kummer surfaces are very interesting objects in number theory,
because they are closely related to modular forms and algebraic curves 
(for example, see Remark \ref{pqrRem}). 
There are various studies on motives of Kummer surfaces now
(for example, see \cite{ILP} or \cite{S}).
Thus, although Kummer surfaces are particular $K3$ surfaces,
they attract a number of researchers.
The purpose of this paper is  to obtain an explicit and simple expression of the Picard-Fuchs system of the family of ${\rm Kum} (A)$ for principally polarized Abelian surfaces $A$.

First, let us  recall the classical Picard-Fuchs equation for the family
$\pi_E: \{E(\lambda) \mid \lambda \in \mathbb{P}^1(\mathbb{C})-\{0,1,\infty \} \} \rightarrow \mathbb{P}^1(\mathbb{C})-\{0,1,\infty \}$
of elliptic curves 
\begin{align}\label{EllipticCurve}
E(\lambda) : w^2=v(v-1)(v-\lambda).
\end{align}
A unique  holomorphic $1$-form $\omega_\lambda $  on $E(\lambda)$ is given 
by $\frac{dv}{\sqrt{v(v-1)(v-\lambda)}}$ up to a constant factor.
By a direct calculation, $\omega_\lambda $ satisfies
$
\lambda (1- \lambda) \frac{\partial^2 \omega_\lambda}{\partial \lambda^2} +(1-2\lambda) \frac{\partial \omega_\lambda}{\partial \lambda} -\frac{1}{4} \omega_\lambda = d\left( \frac{\sqrt{v (v-1) (v-\lambda)}}{2(v-\lambda)^2} \right)
$
(for example, see \cite{SU} Section 2.4).
Since the right hand side of this relation is an exact form, 
the integral $\int_{\gamma_\lambda} \omega_\lambda$ is a solution of the differential equation 
$
\lambda (1- \lambda) \frac{d^2 u}{d \lambda^2} +(1-2\lambda) \frac{d u}{d \lambda} -\frac{1}{4} u = 0,
$
where $\gamma_\lambda$ is a $1$-cycle on $E(\lambda)$.
The integral $\int_{\gamma_\lambda} \omega_\lambda$ is a period integral on $E(\lambda)$.
This differential equation is called the Picard-Fuchs equation for the family $\pi_E$.
This is a special Gauss hypergeometric equation.

The family of Kummer surfaces ${\rm Kum}(A)$ for  principally polarized Abelian surfaces $A$ is regarded as a two-dimensional counterpart of the family $\pi_E$.
We can roughly grasp the Picard-Fuchs system for the family of  ${\rm Kum}(A)$ as follows.
Let $\mathbb{P}^2(\mathbb{C})={\rm Proj}(\mathbb{C}[\xi_1,\xi_2,\xi_3])$ be projective plane.
A double covering of $\mathbb{P}^2(\mathbb{C})$ branched along six lines is defined by the equation
$$
z^2 = \prod_{j=1}^6 (c_{1j} \xi_1+ c_{2j} \xi_2 + c_{3j} \xi_3).
$$
This equation defines a $K3 $ surface.
Such a  surface is called a $K3$ surface of type $(3,6)$.
The period mapping for the family of them is studied in   \cite{MSY}  in detail.
Especially, the Picard-Fuchs system for the family of such $K3$ surfaces is given by a system of relatively simple differential equations of holonomic rank six in four independent variables.
The system is called the hypergeometric system of type $(3,6)$. 
In fact, the family of ${\rm Kum}(A)$ for principally polarized Abelian surfaces $A$ is characterized as a subfamily of the family of $K3$ surfaces of type $(3,6)$.
Precisely, if six branch lines of a $K3$ surface of type $(3,6)$ tangent to a conic in $\mathbb{P}^2(\mathbb{C})$,
then we obtain ${\rm Kum}(A)$.
Conversely, any ${\rm Kum}(A)$ can be attained in such a manner.
Therefore, we can theoretically obtain the Picard-Fuchs system for the family of ${\rm Kum}(A)$, if we restrict  the hypergeometric equation of type $(3,6)$ to the relation corresponding to six tangent lines for a conic.
However, to the best of the author's knowledge, 
it is not easy to complete this procedure explicitly, because the calculations for that are much heavy and complicated.

In this paper, instead of the hypergeometric system of type $(3,6)$,
we will apply another idea and techniques to obtain an explicit and simple expression of the Picard-Fuchs system  for the family of Kummer surfaces ${\rm Kum}(A)$ for  principally polarized Abelian surfaces $A$.
The methods of this paper are mainly based on the two following mathematical facts:
(1) ${\rm Kum}(A)$ has a good Jacobian elliptic fibration;
(2) periods of ${\rm Kum}(A)$ satisfies a GKZ hypergeometric system.
Kumar \cite{Kumar} shows that there are 25 different Jacobian elliptic fibrations on ${\rm Kum} (A)$.
In this paper, we bring up one of them.
The Weierstrass model of the elliptic fibration is given by 
\begin{align}\label{KummerEqIntro}
K(p,q,r): y^2 =x(x+t^2) (x+t^3 +p t^2 + q t +r).
\end{align}
Here, $(p,q,r) \in \mathbb{C}^3$ is a tuple of parameters which gives a deformation of $K3$ surfaces  (for detail, see Section 1.1).
In fact, expressions of period integrals induced from (\ref{KummerEqIntro}) has a good compatibility with the theory of GKZ hypergeometric systems.
By applying that theory, we can see that the period integrals are solutions of a system of linear partial differential equations 
\begin{align}\label{GKZpqrIntro}
\begin{cases}
&\displaystyle q^2 \theta_p \theta_r  u= pr \theta_q (\theta_q -1)u  ,\\
&\displaystyle  p^2 \theta_q(\theta_q + 2 \theta_r) u =  q \theta_p(\theta_p -1) u, \\
&\displaystyle  \theta_p (\theta_p+2\theta_q + 3 \theta_r) u = p \left(\theta_p + 2\theta_q+3\theta_r+\frac{1}{2}\right)^2 u,\\
&\displaystyle  p q \theta_r(\theta_q+2\theta_r)u = r \theta_p \theta_q u.
\end{cases}
\end{align}
This system is corresponding to a certain GKZ system (see Section 2.1).
Here, $\theta_p,\theta_q,\theta_r$ are the Euler operators.
We can prove that the holonomic rank of (\ref{GKZpqrIntro}) is equal to six.
This fact implies that
the system (\ref{GKZpqrIntro}) is not eligible for the Picard-Fuchs system for the family of ${\rm Kum}(A)$,
whose holonomic rank  must be five (see Section 1.2).
It means that we can exhume the Picard-Fuchs system in question as a subsystem of the GKZ hypergeometric system (\ref{GKZpqrIntro}).
In order to obtain it explicitly,
we will use another property of the elliptic fibration (\ref{KummerEqIntro}).
A period integral for the family of (\ref{KummerEqIntro}) has an expression via the Euler integral expression of the Gauss hypergeometric function.
By making full use of such a property,
we will obtain a power series expansion of a period integral (Proposition \ref{PropSer}).
We are able to determine a unique partial linear differential equation,
which satisfies the power series but is not derived from the GKZ system (\ref{GKZpqrIntro}).
This is explicitly described as follows:
\begin{align}\label{NewDiffEqIntro}
&9 q r \theta_p( 1+ 2 \theta_r) u 
-  4 p r \theta_q (2 \theta_q + 3 \theta_r)u 
-  4 p^2 q \theta_r (\theta_q + 2 \theta_r) u \notag \\
&\quad\quad +  4 p^2 r \theta_q (\theta_p + 4 \theta_q + 6 \theta_r) u 
+  p q^2 \theta_r (1+ 16 \theta_q  + 30 \theta_r)u =0.
\end{align}
We can prove that 
the system of linear differential equations (\ref{GKZpqrIntro}) and (\ref{NewDiffEqIntro}) is integrable and of holonomic rank five (Theorem \ref{ThmPFS}).
Therefore, this system is the Picard-Fuchs system for the family of ${\rm Kum}(A)$. 
The singular loci of this system are displayed in Corollary \ref{CorSingular}.
Due to the forms of the loci, we can see that our Picard-Fuchs system is not coming from popular and  well-known hypergeometric systems, like Appell's system or Lauricella's system.
As far as the author can see,
the elliptic fibration corresponding to (\ref{KummerEqIntro}),
which is called Fibration 3 in \cite{Kumar},
 is the unique one which allows us to apply the techniques of GKZ systems and power series.

Here, we give a short comparison with  previous research related to ours.
Griffin and Malmendier  \cite{GM}  study periods of Kummer surfaces ${\rm Kum}(E_1 \times E_2)$ for products of two elliptic curves $E_1$ and $E_2$.
Their results are based on Jacobian elliptic fibrations on ${\rm Kum}(E_1 \times E_2)$.
They obtain several simple expressions of period integrals by using well-known Gauss or Appell hypergeometric functions.
The family of ${\rm Kum}(E_1 \times E_2)$ is a subfamily of our family of the surfaces (\ref{KummerEqIntro}) (see Remark \ref{RemGM}).
Moreover, Cao, Movasati and Yau \cite{CMY} study the Gauss-Manin connection for the family of the algebraic curves of genus two defined by $C: w^2=v^6 + t_2 v^4 + t_3 v^3 + t_4 v^2 + t_5 v + t_6 $.
They explicitly compute the Gauss-Manin connections for four meromorphic $1$-forms $\frac{dv}{w}, \frac{v dv}{w}, \frac{v^3 dv}{w}, \frac{t_2}{2} \frac{v^2 dv}{w} +\frac{v^4 dv}{w}$.
The family of Kummer surfaces ${\rm Kum}({\rm Jac}(C))$ for the Jacobian varieties ${\rm Jac}(C)$ coincides with the family of the surfaces (\ref{KummerEqIntro}).
Therefore, it would be an interesting problem to investigate the relation between our result and that of \cite{CMY}.
Furthermore, Doran, Harder, Movasati and Whitcher \cite{DHMW} study the Picard-Fuchs system for the family of the $K3$ surfaces
$$
y^2 z w -4 x^3 z + 3\alpha_0 x z w^2 +\beta_0 z w^3 + \gamma_0 x z^2 w -\frac{1}{2} (\delta_0 z^2 w^2 + w^4) =0
$$
introduced by \cite{CD}.
Here, $(\alpha_0:\beta_0 :\gamma_0:\delta_0) \in \mathbb{P}(4,6,10,12)$.
They compute the Picard-Fuchs system on each chart $\{\alpha_0\not =0\}, \{\beta_0\not =0\},\{\gamma_0\not =0\}$ and $\{\delta_0\not =0\}$ of the parameter space.
Since the explicit expression of their Picard-Fuchs system is too complicated,
 they give a description of the system only for the chart $\{\alpha_0\not =0\}$ (see \cite{DHMW} Section 5.3).
 The family \cite{CD} is the family of the Shioda-Inose partners of the Kummer surfaces ${\rm Kum}(A)$.
 Additionally, the  space of parameters $(p,q,r)$ of our family of Kummer surfaces (\ref{KummerEqIntro})  gives a covering of the  space of $(\alpha_0:\beta_0 :\gamma_0:\delta_0)$ (see (\ref{pqrT}) and Remark \ref{pqrRem}).
Our Picard-Fuchs system   (\ref{GKZpqrIntro}) and (\ref{NewDiffEqIntro})  may be related to that of \cite{DHMW}.
Our research and that of \cite{DHMW} are based on different motivations and methods.
Each of them has interesting characteristics, respectively. 
The author highlights the following features of our results:
\begin{itemize}
\item 
every Kummer surface ${\rm Kum}(A)$ is attained by (\ref{KummerEqIntro}) for  $(p,q,r)\in \mathbb{C}^3$,
which gives a system of coordinates of only one chart $\mathbb{C}^3$ (see Lemma \ref{LemKum});

\item
our Picard-Fuchs system has a simple expression in terms of $(p,q,r)$;

\item
we have an explicit holomorphic  solution with a power series expression of our Picard-Fuchs system (see Proposition \ref{PropSer}).
\end{itemize}
Our parameter space has a natural and simple compactification as in Lemma \ref{LemKum}.
The characteristics of our results are based on this fact.

The author anticipates that our  family of Kummer surfaces with the explicit Picard-Fuchs system will bring certain benefits to research in various areas.
Sato \cite{S} studies higher Chow cycles on Kummer surfaces ${\rm Kum}(E_1 \times E_2)$ precisely by using appropriate expressions of  periods and the Picard-Fuchs operators for ${\rm Kum}(E_1 \times E_2)$.
Since the family of ${\rm Kum}(E_1 \times E_2)$ can be regarded as a subfamily of our family of the Kummer surface (\ref{KummerEqIntro}),
it appears to the author  that our expressions of the period and the Picard-Fuchs system  for (\ref{KummerEqIntro}) will be useful for further research.
Moreover, a Kummer surface is a typical two-dimensional Calabi-Yau variety.
Since Picard-Fuchs systems for Calabi-Yau varieties are very important in mirror symmetry,
which suggests non-trivial relations between geometry and string theory,
the author expects that our family provides a handy model in such a research field.

\section{Period integrals on Kummer surfaces}

Let $S$ be an algebraic $K3$ surface.
Let $\omega$ be a unique non-zero holomorphic $2$-form on $S$ up to a constant factor.
The $2$-cohomology group $H^2(S,\mathbb{Z})$ is regarded as  the even unimodular lattice $II_{3,19}$ of signature $(3,19)$
via the topological cup product $H^2(S,\mathbb{Z}) \times H^2(S,\mathbb{Z}) \rightarrow \mathbb{Z}$. 
This lattice will be denoted by $L_{K3}$.
The N\'eron-Severi lattice ${\rm NS}(S)$ is  a sublattice of $H^2(S,\mathbb{Z})$ defined by $H^2(S,\mathbb{Z}) \cap H^{1,1}(S,\mathbb{R})$.
This is a non-degenerate lattice of signature $(1,\rho-1)$.
Here, $\rho={\rm rank}({\rm NS}(S))$ is called the Picard number of $S$.
The orthogonal complement ${\rm Tr}(S)$ of ${\rm NS}(S)$ in $H^2(S,\mathbb{Z})$ is called the transcendental lattice.
This is of signature $(2,20-\rho)$.
We can identify $H^2(S,\mathbb{Z})$ with the $2$-homology group $H_2(S,\mathbb{Z})$ by the Poincar\'e duality.
From now on, we often regard ${\rm NS}(S)$ and ${\rm Tr}(S)$ as sublattices of $H_2(S,\mathbb{Z})$.
Then, ${\rm NS}(S)$ in $H_2(S,\mathbb{Z})$ is equal to the kernel of the linear mapping $H_2(S,\mathbb{Z})\ni \gamma \mapsto \int_\gamma \omega\in \mathbb{C}$.

Let $M$ be an even non-degenerate lattice in $L_{K3}$ of signature $(1,\nu)$.
An $M$-polarized $K3$ surface is a pair $(S,j)$ of a $K3$ surface $S$ and a primitive embedding $j: M\hookrightarrow {\rm NS} (S)$.

Let $A=\mathbb{C}^2/\Lambda$ be a principally polarized Abelian surface,
where $\Lambda$ is a lattice of $\mathbb{C}^2$.
Let $(z_1,z_2)$ be the coordinates of $\mathbb{C}^2$.
The Abelian surface $A$ admits an involution $\iota$ derived from $(z_1,z_2) \mapsto (-z_1,-z_2)$ on $\mathbb{C}^2$.
The minimal resolution of the quotient surface $A/\langle \iota \rangle$ is called the Kummer surface ${\rm Kum}(A)$ for $A$. 
It is an algebraic $K3$ surface.

According to \cite{Mo}, the family of Kummer surfaces ${\rm Kum}(A)$ for principally polarized Abelian surfaces $A$ is a family of ${\bf M_{\rm Kum}}$-polarized $K3$ surfaces,
where ${\bf M_{\rm Kum}}$ is an even non-degenerate lattice of signature $(1,16)$
whose orthogonal complement  is explicitly given by
\begin{align}\label{TrKum}
{\bf A_{\rm Kum}} =U(2) \oplus U(2) \oplus A_1(-2)
\end{align}
of signature $(2,3)$.
Here, $U(2)$ ($A_1(-2)$) is the lattice of rank $2$ ($1$, resp.) whose intersection matrix is $\begin{pmatrix} 0 & 2 \\ 2 & 0 \end{pmatrix}$ ($(-4)$ , resp.).
We note that there is a unique primitive embedding $ {\bf M_{\rm Kum}} \hookrightarrow L_{K3}$ up to isometry
(see \cite{NiBilinear} Theorem 1.14.4 or \cite{Mo} Theorem 2.8).

\subsection{Explicit model of Kummer surfaces}

In this paper, we will study a family of complex elliptic surfaces
\begin{align}\label{KummerEq}
K(p,q,r): y^2 =x(x+t^2) (x+t^3 +p t^2 + q t +r),
\end{align}
where $p,q,r$ are complex parameters.
This equation defines an elliptic fibration $(x,y,t) \mapsto t$ with singular fibres of Kodaira type $I_4+6 I_2 + I_2^*$.
As we will see in this subsection, 
the equation (\ref{KummerEq}) gives an explicit model of the Kummer surface ${\rm Kum}(A)$ for a principally polarized Abelian surface $A$.

First, we note the origin of (\ref{KummerEq}).
For a $K3$ surface $S$ over $\mathbb{C}$,
a Jacobian elliptic fibration $S\rightarrow \mathbb{P}^1(\mathbb{C})$ is an elliptic fibration with  a section $\mathbb{P}^1(\mathbb{C}) \rightarrow S$.
Kumar \cite{Kumar} shows that there are 25 different Jacobian elliptic fibrations on the Kummer surface ${\rm Kum}(A)$ for a principally polarized Abelian surface $A$.
He obtains an explicit defining equation for each elliptic fibration.
In particular,
Kumar's Fibration 3 is given by the defining equation
\begin{align}\label{KumarNo3}
y_0^2 = &(x_0 + 4 (\lambda_1 - 1) \lambda_2 (\lambda_3 - \lambda_2) t_0 (t_0 + 4 (\lambda_2 - \lambda_1) (\lambda_3 - 1)))\notag\\
 &\times (x_0 -  4 \lambda_2 (\lambda_2 - \lambda_1) (\lambda_3 - 1) t_0 (t_0 - 4 (\lambda_1 - 1) (\lambda_3 - \lambda_2))) \notag \\
&\times (x_0 - (t_0 - 4 (\lambda_1 - 1) (\lambda_3 - \lambda_2))(t_0 + 4 (\lambda_2 - \lambda_1) (\lambda_3 - 1)) \notag  \\
&\quad \quad\quad\quad\quad \times (\lambda_1 \lambda_3 t_0 + 4 (\lambda_1 - 1) (\lambda_2 - \lambda_1) (\lambda_3 - 1) (\lambda_3 - \lambda_2))).
\end{align}
Here,
$\lambda_1,\lambda_2,\lambda_3$ are complex parameters
which are coming from a defining equation of a hyperelliptic curve $w^2 =v (v-1)(v-\lambda_1)(v-\lambda_2)(v-\lambda_3)$
of genus two.
If we perform a transformation 
$$
x_0 \mapsto x_0 -4 (\lambda_1 - 1) \lambda_2 (\lambda_3 - \lambda_2) t_0 (t_0 + 4 (\lambda_2 - \lambda_1) (\lambda_3 - 1))
$$
to (\ref{KumarNo3}),
we obtain an equation in the form
$$
y_0^2=x_0 \left(x_0+ b_2(\lambda) t_0^2 \right) \left(x_0+c_3(\lambda) t_0^3 + c_2(\lambda) t_0^2 + c_1(\lambda) t_0+ c_0(\lambda) \right),
$$
where $b_2(\lambda)$ and $c_j(\lambda)$ $(j\in\{0,1,2,3\})$ are polynomials in $\lambda_1,\lambda_2,\lambda_3$.
By a transformation $(x_0,y_0,z_0) \mapsto (\mu_x(\lambda)x_0, \mu_y (\lambda)y_0, \mu_z (\lambda) z_0)$,
where $\mu_x(\lambda), \mu_y(\lambda),\mu_z(\lambda)$ are appropriate functions in $\lambda_1, \lambda_2, \lambda_3$,
we obtain (\ref{KummerEq}).
We have a correspondence between $(\lambda_1,\lambda_2,\lambda_3)$ and $(p,q,r)$ given by
\begin{align}\label{pqrlambda}
\begin{cases}
 p=& -\frac{1}{d_\lambda} (\lambda_1 \lambda_2 - \lambda_1^2 \lambda_2 - \lambda_1 \lambda_3 + 2 \lambda_1^2 \lambda_3 - 3 \lambda_1 \lambda_2 \lambda_3 + 2 \lambda_1^2 \lambda_2 \lambda_3\\
     &\quad \quad  + \lambda_2^2 \lambda_3 - \lambda_1 \lambda_2^2 \lambda_3 + 2 \lambda_1 \lambda_3^2- 3 \lambda_1^2 \lambda_3^2 - \lambda_2 \lambda_3^2 + 2 \lambda_1 \lambda_2 \lambda_3^2),\\
 q=&\frac{1}{d_\lambda^2} ( \lambda_1 -1) \lambda_1 (\lambda_1 - \lambda_2) (\lambda_2 - \lambda_3) ( \lambda_3 -1) \lambda_3 (\lambda_1 - \lambda_2 + \lambda_1 \lambda_2 + \lambda_3 - 3 \lambda_1 \lambda_3 + \lambda_2 \lambda_3), \\
 r=& \frac{1}{d_\lambda^3} ( \lambda_1 -1)^2 \lambda_1^2 (\lambda_1 - \lambda_2)^2 (\lambda_2 - \lambda_3)^2 ( \lambda_3 -1 )^2 \lambda_3^2 ,
 \end{cases}
\end{align}
 where $d_\lambda= ( \lambda_2 -1 ) \lambda_2 (\lambda_1 - \lambda_3) $.

Let us consider (\ref{KummerEq}) more precisely.
We will obtain a natural compactification of the parameter space ${\rm Spec}(\mathbb{C}[p,q,r])$.
Let
\begin{align}\label{KummerM}
\overline{K}(p,q,r,b) : y^2 = x(x+b t^2 w^2) (x+t^3 + p t^3 w^2+ q t w^4 +r w^6)
\end{align}
be a hypersurface in the weighted projective space $\mathbb{P}(6,9,2,1) = {\rm Proj} (\mathbb{C}[x,y,t,w])$,
where $(p,q,r,b)\in \mathbb{C}^4-\{0\}$ are complex parameters.
We have an action of the multiplicative group $\mathbb{C}^*$ on $\mathbb{P}(6,9,2,1)$ ($\mathbb{C}^4-\{0\}$, resp.)
given by $(x,y,t,w) \mapsto (x,y,t,\lambda^{-1} w)$ ($(p,q,r,b)\mapsto (\lambda^2 p, \lambda^4 q, \lambda^6 r, \lambda^2 b)$, resp.) for $\lambda \in \mathbb{C}^*$.
The surface $\overline{K}(p,q,r,b)$ is invariant under this action.
Thus, we naturally  have a family
$\left\{\overline{K}(p,q,r,b) \mid (p:q:r:b) \in \mathbb{P}(2,4,6,2) \right\} \rightarrow  \mathbb{P}(2,4,6,2) $
of algebraic surfaces.
The parameter $(p,q,r)$ of (\ref{KummerEq}) is regarded as the point $(p:q:r:1) \in \mathbb{P}(2,4,6,2)$.
 Setting
$\mathcal{T}=\mathbb{P}(2,4,6,2) -\{b=0\}$,
we have a family
\begin{align}\label{KummerFamily}
\varpi: \left\{\overline{K}(p,q,r,b) \mid (p:q:r:b) \in \mathcal{T} \right\} \rightarrow  \mathcal{T}.
\end{align}

\begin{lem}\label{LemKum}
Every Kummer surface ${\rm Kum}(A)$ for a principally polarized Abelian surface $A$ is given by a member of the family $\varpi$ of (\ref{KummerFamily}).
\end{lem}

\begin{proof}
According to \cite{NS} Section 6.1 (see also \cite{NS} (2.3)),
the family of Kummer surfaces ${\rm Kum}(A)$ for a principally polarized Abelian surface $A$ is explicitly given by the family of elliptic surfaces with the Weierstrass equation
\begin{align}\label{KummerCanonical}
y'^2 = x'^3 +({\bf t}_4 s^4 +{\bf t}_{10} s^2) x' +(s^8 + {\bf t}_6 s^6 +{\bf t}_{12} s^4).
\end{align}
Here, the tuple of parameters ${\bf t}_j$ $(j \in \{4,6,10,12\})$ 
corresponds to  a point $({\bf t}_4:{\bf t}_6:{\bf t}_{10}:{\bf t}_{12}) \in \mathbb{P}(4,6,10,12)-\{{\bf t}_{10}={\bf t}_{12}=0\}$.
We remark that the fibration (\ref{KummerCanonical})  coincides with the Fibration 13 of \cite{Kumar} under an appropriate transformation of parameters.
On the other hand, let us consider the elliptic surface
\begin{align}\label{KummerM1}
y^2 = x(x+b t^2 ) (x+t^3 + p t^2 + q t  +r ).
\end{align}
This is derived from (\ref{KummerM}), if we put $w=1$. 
In this proof,
we will give a birational transformation from (\ref{KummerM1})  to (\ref{KummerCanonical}) explicitly
and show that every surface of (\ref{KummerCanonical}) is attained by the surface (\ref{KummerM1}).

Perform a birational transformation 
\begin{align*}
x=\frac{1}{2} (- s_0^2 - b t^2 - y_0),
\quad
y=-\frac{1}{2} s_0 (-2 r +  s_0^2 - 2 q t + b t^2 - 2 p t^2 - 2 t^3 + y_0)
\end{align*}
to (\ref{KummerM1}).
Then, we have an elliptic surface
\begin{align}\label{EllQuad}
y_0^2= b^2 t^4- 4 s_0^2 t^3 + 2 b s_0^2 t^2 - 4 p s_0^2 t^2 - 4 q s_0^2 t -4 r s_0^2 + s_0^4,
\end{align}
whose right hand side is a polynomial of degree $4$ in $t$.
We can transform (\ref{EllQuad}) into the Weierstrass equation
\begin{align}\label{Weierstrass0}
y_0'^2= &x_0'^3+ \left(b^2 r s_0^2 - \frac{b^2 s_0^4}{3} + \frac{b p s_0^4}{3}  - \frac{p^2 s_0^4}{3} +  q s_0^4 \right) x \notag\\
&+\left(\frac{b^2 q^2 s_0^4}{4} + \frac{b^3 r s_0^4}{3} - \frac{2b^2 p r s_0^4 }{3}  - \frac{ 2 b^3 s_0^6}{27} + \frac{b1^2 p s_0^6}{9}  + \frac{b p^2 s_0^6}{9}  - \frac{ 2 p^3 s_0^6}{27} - \frac{b q s_0^6}{6}  + \frac{p q s_0^6}{3}  - r s_0^6 + \frac{s_0^8}{4}\right)
\end{align}
by applying a technique appeared in \cite{AKMMMP} Section 3.1.
By putting $s_0=2s$, $x_0'=4 x'$, $y_0'=8 y'$  to (\ref{Weierstrass0}), 
we obtain (\ref{KummerM1}) whose coefficients ${\bf t}_j$ $(j \in\{4,6,10,12\} )$ are explicitly given by weighted homogeneous polynomials in $p,q,r,b$ as follows:
\begin{align}\label{pqrT}
\begin{cases}
&\displaystyle {\bf t}_4(p,q,r,b)=\left(-\frac{b^2}{3} + \frac{b p}{3} - \frac{p^2}{3} + q\right), \quad {\bf t}_6(p,q,r,b)=-\frac{1}{54} (b - 2 p) (4 b^2 + 2 b p - 2 p^2 + 9 q) -r,  \vspace{2mm} \\
&\displaystyle {\bf t}_{10}(p,q,r,b)=\frac{1}{4} b^2 r,  \quad {\bf t}_{12}(p,q,r,b)=\frac{1}{48} b^2 (3 q^2 + 4 b r - 8 p r). \\
\end{cases}
\end{align}
We have the mapping $g: \mathbb{P}(2,4,6,2) \rightarrow \mathbb{P}(4,6,10,12)$ given by
$(p:q:r:t) \mapsto ({\bf t}_4(p,q,r,b):  {\bf t}_6(p,q,r,b): {\bf t}_{10}(p,q,r,b) :{\bf t}_{12}(p,q,r,b)) $.
For every point $({\bf t}_4:  {\bf t}_6: {\bf t}_{10} :{\bf t}_{12}) \in \mathbb{P}(4,6,10,12)-\{{\bf t}_{10}={\bf t}_{12}=0\}$,
we can check that $g^{-1} ({\bf t}_4:  {\bf t}_6: {\bf t}_{10} :{\bf t}_{12}) \cap \mathcal{T} \not = \phi$
In fact, for a generic point $({\bf t}_4:  {\bf t}_6: {\bf t}_{10} :{\bf t}_{12}) \in \mathbb{P}(4,6,10,12)-\{{\bf t}_{10}={\bf t}_{12}=0\}$,
$g^{-1} ({\bf t}_4:  {\bf t}_6: {\bf t}_{10} :{\bf t}_{12}) \cap \mathcal{T}$ consists of $20$ points.
It shows that every Kummer surface ${\rm Kum}(A)$ can be attained by the surface (\ref{KummerM1}).
\end{proof}

The above proof is based on a technique appeared in \cite{Kumar},
by which Kumar transfers an elliptic fibration to another one.
The correspondence (\ref{pqrT}) of parameters shows a feature of our  compactification of the parameter space.
Also,
there does not appear  precise descriptions of the birational transformations  in \cite{Kumar}.
Therefore,
the author  gives an explicit proof of the above lemma here.

\begin{rem}\label{pqrRem}
We have an explicit correspondence
$
{\bf t}_4=-3\alpha_0, {\bf t}_6=-2\beta_0,  {\bf t}_{10}=-\gamma_0$ and ${\bf t}_{12}=\delta_0$
between the parameters $({\bf t}_4:{\bf t}_6:{\bf t}_{10}:{\bf t}_{12})$ and $(\alpha_0:\beta_0:\gamma_0:\delta_0)$ ,
which are the parameters of the Clingher-Doran family due to \cite{CD} appeared in Introduction.
We note that the parameters ${\bf t}_4,  {\bf t}_6, {\bf t}_{10},{\bf t}_{12}$ in (\ref{KummerCanonical}) have an expression in terms of Siegel modular forms of degree two via the period mapping for the family of Kummer surfaces.
Furthermore, 
$(\lambda_1,\lambda_2,\lambda_3)$ in (\ref{pqrlambda}) has an explicit expression in terms of  the Riemann theta constants.
The parameters ${\bf t}_4,  {\bf t}_6, {\bf t}_{10}$ and ${\bf t}_{12}$ are described by  $\lambda_1,\lambda_2$ and $\lambda_3$  via the Igusa-Clebsch invariants.
For detail, one can refer to \cite{NS1}.
According to (\ref{pqrlambda}) and (\ref{pqrT}),
the tuple of our parameters $(p,q,r)$ are regarded as interpolations between $({\bf t}_4:  {\bf t}_6: {\bf t}_{10} : {\bf t}_{12})$ and  $(\lambda_1,\lambda_2,\lambda_3)$. 
\end{rem}

\subsection{Period mapping for our family of Kummer surfaces}

Set $\mathcal{W} = {\rm Spec}(\mathbb{C}[p,q,r]) - \{q=r=0\}$.
According to Lemma \ref{LemKum}, together with (\ref{pqrT}),   we have a family
\begin{align}\label{KummerFamily1}
\pi: \mathcal{G}=\{K(p,q,r) \mid (p,q,r) \in \mathcal{W} \} \rightarrow  \mathcal{W}
\end{align}
of Kummer surfaces.

Take a generic point $(p_0,q_0,r_0)\in \mathcal{W}$ such that
$K_0= K(p_0,q_0,r_0) $ satisfies ${\rm NS}(K_0)={\bf M}_{\rm Kum}$ and ${\rm Tr}(K_0)={\bf A}_{\rm Kum}$.
We identify $H_2(K_0,\mathbb{Z})$ with the $K3$ lattice $L_{K3}$.
Let $\left\{ \gamma_6,\ldots, \gamma_{22} \right\}$ be a basis of ${\rm NS}(K_0)$.
Since ${\bf M_{\rm Kum}}$ is a primitive sublattice in $L_{K3}$,
there are $\gamma_1,\ldots,\gamma_5 \in H_2(K_0,\mathbb{Z}) $ such that $\{\gamma_1,\ldots,\gamma_{22}\}$ is a basis of $H_2(K_0,\mathbb{Z}) \simeq L_{K3}$.
Let $\{\delta_1,\ldots,\delta_{22}\}$ be its dual basis with respect to the unimodular lattice $H_2(K_0,\mathbb{Z})$. 
Then, the intersection matrix of the  sublattice $\langle \delta_1,\ldots,\delta_5 \rangle_\mathbb{Z}$ is given by ${\bf A_{\rm Kum}}$ of (\ref{TrKum}).
Let $\mathcal{U}$ be a sufficiently small neighborhood of $(p_0,q_0,r_0)$ in $\mathcal{W}$.
There exists a topological trivialization $\tau:\{K(p,q,r)\mid (p,q,r) \in \mathcal{U}\} \rightarrow K_0\times \mathcal{U}$. 
Letting $\beta : K_0 \times \mathcal{U} \rightarrow K_0$ be a projection, we put $r = \beta \circ \tau.$
Then, $r'_{(p,q,r)} = r |_{K(p,q,r)}$ gives a $\mathcal{C}^\infty$ isomorphism of complex surfaces.
Hence, we have an isometry $\psi_{(p,q,r)}: H_2(K(p,q,r),\mathbb{Z}) \rightarrow H_2(K_0,\mathbb{Z})\simeq L_{K3}$.
Let 
\begin{align}\label{omega}
\omega_{(p,q,r)}=\frac{dx \wedge dt}{\sqrt{x (x+t^2) (x+t^3 +p t^2 + q t +r)}}
\end{align}
be a unique holomorphic $2$-form on $K(p,q,r)$ up to a constant factor.
If $\Gamma \in {\rm NS}(K_0)$, then $\psi_{(p,q,r)}^{-1}(\Gamma) \in {\rm NS}(K(p,q,r)).$ 
Therefore, we have five non-trivial period integrals for $K(p,q,r)$ on $\psi_{(p,q,r)}^{-1} (\gamma_j)$ $(j\in \{1\ldots,5\})$.
Thus, we have a local period mapping
\begin{align}\label{PeriodPhi}
\mathcal{U} \ni (p,q,r) \mapsto \Phi_\mathcal{U} (p,q,r)=\left( \int_{\psi_{(p,q,r)}^{-1} (\gamma_1)} \omega_{(p,q,r)} : \cdots :   \int_{\psi_{(p,q,r)}^{-1} (\gamma_5)} \omega_{(p,q,r)}\right)
\end{align}
on $\mathcal{U}$ for our family $\pi$ of (\ref{KummerFamily1}).
By virtue of the Riemann-Hodge relation,
$ \Phi_\mathcal{U} (p,q,r)$
is a point of 
$\mathcal{D}_{\rm Kum} = \left\{ \xi\in\mathbb{P}^4 (\mathbb{C}) \mid \xi {\bf A_{\rm Kum}} {}^t \xi =0,  \xi {\bf A_{\rm Kum}} \overline{{}^t \xi} >0 \right\}$.
We have two connected components of $\mathcal{D}_{\rm Kum}$.
Such a connected component is a three-dimensional bounded symmetric domain of type $IV$,
which is biholomorphic to the Siegel upper half plane of degree two.

Let the notation be as above.
Let us consider the locally constant sheaf $R_2 \pi_* \mathbb{Z}$.
 For $(p,q,r) \in \mathcal{U}$,
$\{\pi^{-1}_{(p,q,r)} (\gamma_1),\ldots,\pi^{-1}_{(p,q,r)} (\gamma_{22})  \}$ gives a basis of 
$R_2 \pi_* \mathbb{Z} |_\mathcal{U} = \bigcup_{(p,q,r) \in \mathcal{U}} H_2 (K(p,q,r), \mathbb{Z})$.
Let $\omega_{\mathcal{G}/\mathcal{W}} $ be the sheaf of relative holomorphic $2$-forms for $\pi$.
The holomorphic $2$-form $\omega$ of (\ref{omega}) defines a unique non-trivial section of $H^0 (\mathcal{U}, \pi_* \omega_{\mathcal{G}/\mathcal{W}}) = H^0 (\pi^{-1} (\mathcal{U}), \omega_{\mathcal{G}/\mathcal{W}})$ up to a constant factor.
The right hand side of (\ref{PeriodPhi}) gives the expression of $\omega \in H^0 (\mathcal{U}, \pi_* \omega_{\mathcal{G}/\mathcal{W}})$.
The Gauss-Manin connection $\nabla$ for the family $\pi$, 
which satisfies $\nabla \omega =0$,
 derives a system of linear differential equations in independent variables $p,q,r$ of rank five,
 such that its space of solutions is generated by the integrals appeared in (\ref{PeriodPhi}).
 We will call this system the Picard-Fuchs system for the family $\pi$ of Kummer surfaces.

\subsection{Power series expansion of period integral for Kummer  surfaces}

Let $(\alpha,N)$ be the Pochhammer symbol for $\alpha\in \mathbb{C}$.
The Gauss hypergeometric series is defined as
\begin{align}\label{GHGF}
{}_2 F_1\left(\alpha,\beta,\gamma ; \lambda \right) =\sum_{N=0}^\infty \frac{(\alpha,N)(\beta,N)}{(\gamma,N) N!} \lambda^N,
\end{align}
where $\alpha,\beta\in \mathbb{C}$ and $\gamma\in \mathbb{C}-\mathbb{Z}_{<0}$.
The radius of convergence of the right hand side of (\ref{GHGF}) is $1$.
The theory of elliptic integrals shows that
the Gauss hypergeometric series 
${}_2 F_1\left(\frac{1}{2},\frac{1}{2},1 ; \lambda \right) $
 for $\alpha=\beta=\frac{1}{2}$ and $\gamma=1$
has an expression
\begin{align}\label{EllipticHG}
{}_2 F_1\left(\frac{1}{2},\frac{1}{2},1 ; \lambda\right)  = {\rm const} \int_{\Delta_{v}} \frac{dv}{\sqrt{v (v-1) (v-\lambda)}},
\end{align}
for an appropriate $1$-cycle $\Delta_v$ on the elliptic curve (\ref{EllipticCurve}). 

\begin{prop}\label{PropSer}
Take an appropriate $2$-cycle $\Delta$ on the Kummer surface (\ref{KummerEq}).
Then, a period integral on $\Delta$ has a power series expansion in $p,q,r$ as
\begin{align}\label{KummerPerSer}
&\iint_\Delta \frac{dx \wedge dt}{ \sqrt{x(x+t^2) (x+t^3+p t^2 +q t +r)}} \notag \\
&={\rm const}\sum_{\ell,m,n=0}^\infty \frac{1}{2^{4 (\ell + 2 m + 3 n)}} \frac{(2 (\ell + 2 m + 3 n) )!)^2}{((\ell + 2 m + 3 n)!)^3} \frac{1}{\ell! m! n! (m + 2 n)!} p^\ell q^m r^n.
\end{align}
This expression is valid for $(p,q,r) \in U$,
where $U$ is a sufficiently small neighborhood of the origin $(p,q,r)=(0,0,0)$.  
\end{prop}

\begin{proof}
Suppose that $(p,q,r)$ is a point of an open set $U$ in $(p,q,r)$-space.
Let $\Delta$ be a $2$-cycle on the surface (\ref{KummerEq}).
We have 
\begin{align}\label{OnDelta}
\iint_\Delta \frac{dx \wedge dt}{\sqrt{x(x+t^2)(x+t^3 +p t^2 +q t +r)}} 
={\rm const} \iint_{\Delta_0} \frac{dx \wedge dt}{t^2 \sqrt{x(x-1)\left(x-\left(t+p+\frac{q}{t} +\frac{r}{t^2}\right)\right)}} 
\end{align}
under the transformation $x\mapsto -t^2 x$. 
Here, $\Delta_0$ is the corresponding integral contour to $\Delta$ under the transformation.
In this proof, we will take $\Delta_0$ appropriately and show that the right hand side of (\ref{OnDelta}) has the expression (\ref{KummerPerSer}).
 
We take  $\Delta_0$ as a direct product of $\Delta_x$ and $\Delta_t$,
where $\Delta_x$ ($\Delta_t$, resp.) is a closed arc in $x$-plane ($t$-plane, resp.).
Here, we suppose $\Delta_t$ is an arc given by
$\left\{r e^{\sqrt{-1} \theta} \mid 0\leq \theta <2 \pi \right\}$,
where $r$ satisfies $\varepsilon<r<1$ for a sufficiently small positive number $\varepsilon>0$. 
 By (\ref{EllipticHG}),
 we suppose $\Delta_x$ is an arc such that
\begin{align}\label{OnDeltax}
\int_{\Delta_x} \frac{dx}{\sqrt{x(x-1)\left(x-\left(t+p+\frac{q}{t} +\frac{r}{t^2}\right)\right)}} = {\rm const} \cdot {}_2 F_1 \left(\frac{1}{2},\frac{1}{2},1 ; t+p+\frac{q}{t}+\frac{r}{t^2} \right)
\end{align}
holds for  $t\in \Delta_t$.
We obtain
\begin{align}\label{PResidue}
\iint_\Delta \frac{dx \wedge dt}{\sqrt{x(x+t^2)(x+t^3 +p t^2 +q t +r)}} 
={\rm const} \int_{\Delta_t} \frac{1}{t} \cdot {}_2F_1 \left(\frac{1}{2},\frac{1}{2},1 ; t+p+\frac{q}{t}+\frac{r}{t^2} \right) dt
\end{align}
by (\ref{OnDelta}) and (\ref{OnDeltax}).
Let us recall (\ref{GHGF}).  
We can assume that the neighborhood $U$ is  small enough comparing with every point $t\in \Delta_t$.
Then, we have the power series expansion
\begin{align}\label{SGHGF}
{}_2F_1 \left(\frac{1}{2},\frac{1}{2},1 ; t+p+\frac{q}{t}+\frac{r}{t^2} \right) 
=\sum_{N=0}^\infty \frac{(\frac{1}{2},N)^2}{(N!)^2} \left( t+p+\frac{q}{t}+\frac{r}{t^2} \right)^N,
\end{align}
which converges absolutely uniformly.
Remark that
$\displaystyle \left(\frac{1}{2},N\right) =\frac{1}{2^{2N}} \frac{(2N)!}{n!}$
holds.
Also, we have
\begin{align}\label{Niko}
&\left( t+p+\frac{q}{t} +\frac{r}{t^2} \right)^N
=\sum_{n=0}^N \begin{pmatrix} N  \\ N-n \end{pmatrix} \left( t+p+\frac{q}{t} \right)^{N-n} \left(\frac{r}{t^2}\right)^n \notag\\
&=\sum_{n=0}^N \sum_{m=0}^{N-n} \begin{pmatrix} N  \\ N-n \end{pmatrix} \begin{pmatrix} N-n  \\ N-n-m \end{pmatrix} \left( t+p \right)^{N-m-n}  \left(\frac{q}{t} \right)^m  \left(\frac{r}{t^2}\right)^n \notag\\
&=\sum_{n=0}^N \sum_{m=0}^{N-n} \sum_{\ell=0}^{N-n-m} \begin{pmatrix} N  \\ N-n \end{pmatrix} \begin{pmatrix} N-n  \\ N-n-m \end{pmatrix} \begin{pmatrix} N-n-m  \\ N-n-m-\ell \end{pmatrix} t^{N-n-m-\ell} p^\ell  \left(\frac{q}{t} \right)^m  \left(\frac{r}{t^2}\right)^n \notag\\
&=\sum_{n=0}^N \sum_{m=0}^{N-n} \sum_{\ell=0}^{N-m-n}  \frac{N!}{n! m! \ell! (N-n-m-\ell)!} t^{N-3n-2m-\ell} p^\ell q^m  r^n.
\end{align}
The integral  of (\ref{PResidue})  is calculated by applying the residue theorem and (\ref{Niko}).

Summarizing the above argument, we see that  the left hand side of (\ref{KummerPerSer})  has the expression  
\begin{align*}
\sum_{n=0}^\infty \sum_{m=0}^\infty \sum_{\ell=0}^\infty \frac{1}{2^{4(\ell + 2m + 3n)}} \frac{((2(\ell +2m + 3n))!)^2}{((\ell +2m + 3n)!)^3} \frac{1}{\ell ! m! n! (m+2n) !} p^\ell q^m r^n
\end{align*}
up to a constant factor.
\end{proof}

In the statement of Proposition \ref{PropSer},
although we can evaluate the  constant factor of the right hand side,
we do not give the precise expression of it.
It is not necessary for our argument below.

\begin{rem}\label{RemGM}
Griffin and Malmendier \cite{GM} 
study periods  for Kummer surfaces ${\rm Kum}(E_1 \times E_2)$ for products of two elliptic curves.
They study various elliptic fibrations on ${\rm Kum}(E_1 \times E_2)$
and obtain simple expressions of periods in terms of  the well-known Gauss or Appell hypergeometric functions.
In our paper,
we obtain more complicated expressions of a period in the above proposition for our family $\pi$ of all Kummer surfaces ${\rm Kum}(A)$ for principally polarized Abelian surfaces $A$.
We remark that the family of ${\rm Kum}(E_1 \times E_2)$ is corresponding to the subfamily of our family $\pi$ of (\ref{KummerFamily1}) restricted to the locus $\{r=0\}$. 
This is guaranteed by (\ref{pqrT}) and the fact that the Kummer surface (\ref{KummerCanonical}) degenerates to ${\rm Kum} (E_1 \times E_2)$ if and only if ${\bf t}_{10}=0$ (see \cite{NS}).
\end{rem}

\section{Explicit expression of Picard-Fuchs system for Kummer surfaces}

In this section, we will obtain an explicit description of the Picard-Fuchs system of the family of all Kummer surfaces for principally polarized Abelian surfaces as a subsystem of a certain GKZ hypergeometric equations.

\subsection{Application of GKZ hypergeometric systems to Kummer surfaces}

We start this subsection with a short survey of the theory of GKZ hypergeometric systems (\cite{GKZ}).
For an independent variable $\lambda$,  
let $ \theta_\lambda$ be the  Euler operator: $ \theta_\lambda=\lambda \frac{\partial}{\partial \lambda} $.
In this subsection, we regard $\mathbb{C}^n $ as the vector space of column vectors with $n$ entries over $\mathbb{C}$.

Let $A$ be an $(m+k)\times n$ matrix given by the form
\begin{align}\label{MatrixA}
A=
\begin{pmatrix}
1 &\cdots & 1 &0 &\cdots &0 &\cdots &0 &\cdots &0 \\  
0& \cdots&0 & 1 &\cdots & 1 & \cdots &0 &\cdots &0  \\
\vdots& &\vdots  &\vdots &     &\vdots &  &\vdots &  &\vdots  \\  
0&\cdots &0   &0&\cdots &0  & \cdots   & 1 &\cdots & 1 \\
a_1'& \cdots & a_{\ell_1}' & a_{\ell_1+1}' & \cdots & a_{\ell_2}' & \cdots & a_{\ell_{m-1}+1}' &\cdots& a_{\ell_m}'
\end{pmatrix},
\end{align}
where $a_j'=\begin{pmatrix} a_{1j} \\  \cdots \\ a_{kj}\end{pmatrix} \in \mathbb{C}^k$ and $\ell_m=n$. 
Also, we take
\begin{align}\label{VectorGamma}
\gamma
={}^t(\alpha_1 , \cdots , \alpha_m ,-\beta_1-1 , \cdots , -\beta_k -1) \in \mathbb{C}^n.
\end{align}
For $t=(t_1,\ldots,t_k)$, we set 
$t^{a_j'} = t_1^{a_{1j}} \cdots t_k^{a_{kj}}$.
Moreover, we set
$$
P_i (t; c) = \sum_{j=\ell_{i-1}+1}^{\ell_i} c_j t^{a_j'} \quad \quad (i\in\{1,\ldots,m\})
$$
for $c=(c_1,\ldots,c_n)$.
Here, $\ell_0$ stands for $0$.
Letting $\Delta$ be a twisted cycle, we have an integral
\begin{align}\label{IntGKZ}
F_\Delta (\alpha,\beta; P_1,\ldots,P_m)
=\int_{\Delta} P_1(t;c)^{\alpha_1} \cdots P_m(t;c)^{\alpha_m} t_1^{\beta_1} \cdots t_k^{\beta_k} dt_1 \wedge \cdots \wedge dt_k.
\end{align}
By virtue of \cite{GKZ},
we obtain a system of linear differential equations in the independent variables $c_1,\ldots, c_n$ satisfying the integral (\ref{IntGKZ}). 
From now on, we put
$\displaystyle \theta_j=c_j \frac{\partial}{\partial c_j}$.

\begin{prop}\label{PropGKZ}
Let the notation be as above.
The integral (\ref{IntGKZ}) is a solution of a holonomic system of linear differential equations
\begin{align}
&\displaystyle \sum_{j=\ell_{i-1}+1}^{\ell_i} \theta_j u =\alpha u,\label{GKZ1}\\
&\displaystyle \sum_{j=1}^n a_{\kappa  j } \theta_j =(-\beta_\kappa -1) u, \label{GKZ2}\\
&\displaystyle \prod_{j: b_j>0} \left(\frac{\partial}{\partial c_j}\right)^{b_j} u =  \prod_{j: b_j<0} \left(\frac{\partial}{\partial c_j}\right)^{-b_j} u \quad\quad\quad (b={}^t (b_1,\ldots,b_n) \in {\rm Ker}(A) \cap \mathbb{Z}^n).  \label{GKZ3}
\end{align}
\end{prop}

The above system is a particular GKZ hypergeometric system concordant with  the integral of (\ref{IntGKZ}).

\begin{rem} \label{RemRankGKZ}
According to the argument of  the original paper \cite{GKZ},
the equations (\ref{GKZ1}) and (\ref{GKZ3}) satisfy the integrand $\prod_{j=1}^m P_j(t; c)^{\alpha_j} \prod_{\kappa=1}^k t_\kappa^{\beta_\kappa}$  rather than its integral.  
We need a more delicate argument for (\ref{GKZ2}).
Let $\omega$ be the $k$-form such that $F_\Delta (\alpha,\beta; P_1,\ldots,P_m) =\int_\Delta \omega$.
For $\kappa \in \{1,\ldots,k\}$,
we can see that $(\sum_{j=1}^m a_{\kappa j} \theta_j +(\beta_\kappa +1))\omega$ is calculated to be
a special exact differential form given in the form of
$\eta_\kappa =d (\psi_\kappa dt_1 \wedge \ldots \wedge dt_{\kappa-1} \wedge dt_{\kappa+1} \wedge \ldots \wedge dt_k)$
for a certain function $\psi_\kappa$.
Since $\int_\Delta \eta_\kappa =0$, (\ref{GKZ2}) has a solution $u =\int_\Delta \omega$.
However, the GKZ hypergeometric equations  (\ref{GKZ1}), (\ref{GKZ2}) and (\ref{GKZ3})  do not attain all differential equations satisfying the integral of (\ref{IntGKZ}).
\end{rem}

We will apply the above theory  to periods for the family $\pi$ of   (\ref{KummerFamily1}).
Let us study a particular system of GKZ hypergeometric equations associated with
\begin{align}\label{AGammaK}
A_K
=\begin{pmatrix}
1& 1&0 &0 &0 &0 &0\\
0&0 &1& 1& 1& 1& 1 \\
1& 0&1&0&0&0&0 \\
0&2 &0&3&2&1&0
\end{pmatrix},
\quad\quad
\gamma_K=\begin{pmatrix} \frac{-1}{2} \\ \frac{-1}{2} \\ \frac{-1}{2} \\ -1 \end{pmatrix}.
\end{align}
From Proposition \ref{PropGKZ}, such a system has a solution 
\begin{align}\label{IntGKZKummer}
\int_\Delta t_1^{-\frac{1}{2}} (c_1 t_1 + c_2 t_2^2)^{-\frac{1}{2}}(c_3 t_1 + c_4 t_2^3 + c_5 t_2^2 +c_6 t_2 +c_7)^{-\frac{1}{2}} dt_1 \wedge dt_2.
\end{align}
Let us put
\begin{align}\label{pqr}
p=\frac{c_1 c_5}{c_2 c_3}, \quad q=\frac{c_1^2 c_4 c_6}{c_2^2 c_3^2} ,\quad r=\frac{c_1^3 c_4^2 c_7}{ c_2^3 c_3^3}.
\end{align}
Also, we put
$\displaystyle t_1=\frac{c_2^3 c_3^2}{c_1^3 c_4^2} x$ and $\displaystyle t_2=\frac{c_2 c_3}{c_1 c_4} t$.
We can see that this integral (\ref{IntGKZKummer}) derives the left hand side of (\ref{KummerPerSer}). 
We obtain
\begin{align}\label{thetapqr}
\theta_5=\theta_p,\quad \theta_6=\theta_q,\quad \theta_7=\theta_r
\end{align}
from (\ref{pqr}).
The differential equations  of (\ref{GKZ1}) and (\ref{GKZ2})  induce 
\begin{align}\label{PGKZ12}
\begin{cases}
&(\theta_1 + \theta_2 +\frac{1}{2}  )u=0, \\
& (\theta_3 +\theta_4+\theta_5 + \theta_6 +\theta_7 +\frac{1}{2})u =0, \\
& (\theta_1+\theta_3+\frac{1}{2}) u=0,\\
& (2\theta_2 +3\theta_4+ 2\theta_5 + \theta_6 +1) u=0.
\end{cases}
\end{align}
Hence, together with (\ref{pqr}),
we have the relations
\begin{align}\label{theta1234}
\begin{cases}
& \theta_1 u =(\theta_p + 2 \theta_q + 3 \theta_r) u \\
& \theta_2 u =\theta_3 u = -(\theta_p + 2 \theta_q + 3 \theta_r+\frac{1}{2} )u,\\
& \theta_4 u =(\theta_q+ 2 \theta_r) u.
\end{cases}
\end{align}
Next,
let us consider a holonomic system (\ref{GKZ3}) defined by $b\in {\rm Ker}(A) \cap \mathbb{Z}^4$. 
We study four differential equations
\begin{align}\label{PGKZpqr}
 \frac{\partial^2}{\partial c_5 \partial c_7} u = \frac{\partial^2}{\partial c_6^2} u , \hspace{3mm}
 \frac{\partial^2}{\partial c_4 \partial c_6} u = \frac{\partial^2}{\partial c_5^2} u , \hspace{3mm}
 \frac{\partial^2}{\partial c_1 \partial c_5} u = \frac{\partial^2}{\partial c_2 \partial c_3} u ,  \hspace{3mm}
 \frac{\partial^2 }{\partial c_4 \partial c_7} u = \frac{\partial^2 }{\partial c_5 \partial c_6} u
\end{align}
of order two,
determined by four vectors
$$
{}^t(0,0,0,0,1,-2,1), \hspace{2mm}  {}^t(0,0,0,1,-2,1,0),  \hspace{2mm}  {}^t(1,-1,-1,0,1,0,0),  \hspace{2mm}  {}^t(0,0,0,1,-1,-1,1) \in {\rm Ker}(A_K)\cap \mathbb{Z}^7. 
$$
Any other equations (\ref{GKZ3}) induced from  ${\rm Ker}(A_K)$  of (\ref{AGammaK})  are attributed to (\ref{PGKZpqr}).
Also, we remark that the relation
$\displaystyle \frac{\partial^2}{\partial c_j^2} =\theta_j (\theta_j -1)$ holds.
By using (\ref{pqr}), (\ref{thetapqr}) and (\ref{theta1234}),
the equations (\ref{PGKZpqr}) give
\begin{align}\label{GKZpqr}
\begin{cases}
&\displaystyle q^2 \theta_p \theta_r u= pr \theta_q (\theta_q -1)u  ,\\
&\displaystyle  p^2 \theta_q(\theta_q + 2 \theta_r) u =  q \theta_p(\theta_p -1) u, \\
&\displaystyle  \theta_p (\theta_p+2\theta_q + 3 \theta_r) u = p \left(\theta_p + 2\theta_q+3\theta_r+\frac{1}{2}\right)^2 u,\\
&\displaystyle  p q \theta_r(\theta_q+2\theta_r)u = r \theta_p \theta_q u.
\end{cases}
\end{align}
For example,
the second equation of (\ref{PGKZpqr}) is equal to 
$$
c_5^2 \theta_4 \theta_6 u = c_4 c_6 \theta_5(\theta_5 -1) u.
$$
By multiplying the both sides by $\displaystyle \frac{c_1^2}{c_2^2 c_3^2}$, 
we obtain the second equation of (\ref{GKZpqr}).

\subsection{Differential equation for Kummer surfaces beyond GKZ hypergeometic system}

As we saw in Remark \ref{RemRankGKZ}, 
the equations in (\ref{GKZpqr}) do not give all differential equations satisfying the period integrals for the family $\pi$ of (\ref{KummerFamily1}).
We can eliminate four of six operators $\theta_p^2 , \theta_q^2 , \theta_r^2 , \theta_p \theta_q  , \theta_q \theta_r  ,  \theta_r \theta_p  $
of rank two
by using the relations of (\ref{GKZpqr}).
We can directly check that 
the  system (\ref{GKZpqr}) is of order six, by calculating an integrable Pfaffian system for a basis, i.e. 
$\{u,\theta_p u, \theta_q u, \theta_r u, \theta_{p}^2 u, \theta_{q}^2 u\}$.

However, 
as we saw in Section 1.2,
the Picard-Fuchs system for the family $\pi$ of Kummer surfaces must be of holonomic rank five.
We need to obtain an equation beyond the GKZ hypergeometric system (\ref{GKZpqr})
in order to obtain exact Picard-Fuchs system.
We will obtain a new differential equation of order two
by making full use of the explicit power series expression (\ref{KummerPerSer}) of the period for $\pi$.

\begin{thm}\label{ThmPFS}
(1) The period integral of (\ref{KummerPerSer}) satisfies a partial linear differential equation
\begin{align}\label{NewDiffEq}
&9 q r \theta_p( 1+ 2 \theta_r) u 
-  4 p r \theta_q (2 \theta_q + 3 \theta_r)u 
-  4 p^2 q \theta_r (\theta_q + 2 \theta_r) u \notag \\
&\quad\quad +  4 p^2 r \theta_q (\theta_p + 4 \theta_q + 6 \theta_r) u 
+  p q^2 \theta_r (1+ 16 \theta_q  + 30 \theta_r)u =0.
\end{align}

(2) The system of partial linear differential equations (\ref{GKZpqr}) and (\ref{NewDiffEq}) gives a Picard-Fuchs system for the family $\pi$  of  (\ref{KummerFamily1}).
\end{thm}

\begin{proof}
(1) The equation (\ref{NewDiffEq}) is found by a method of indeterminate coefficients.
Let $\mathcal{P}(p,q,r)$ be the power series of the right hand side of  (\ref{KummerPerSer}). 
As we saw in the end of Section 2.1,
$u=\mathcal{P}(p,q,r)$ satisfies every equation of (\ref{GKZpqr}).
Let us consider a linear differential equation in the form
\begin{align*}
\sum_{\alpha,\beta,\gamma \geq 0} p^\alpha q^\beta r^\gamma (&a_{000}^{\alpha,\beta,\gamma} + a_{100}^{\alpha,\beta,\gamma} \theta_p  + a_{010}^{\alpha,\beta,\gamma} \theta_q + a_{001}^{\alpha,\beta,\gamma} \theta_r  \notag\\
&+a_{200}^{\alpha,\beta,\gamma}  \theta_{p}^2  
+a_{020}^{\alpha,\beta,\gamma} \theta_q^2 +a_{002}^{\alpha,\beta,\gamma} \theta_r^2 
+a_{110}^{\alpha,\beta,\gamma} \theta_p \theta_q +a_{011}^{\alpha,\beta,\gamma} \theta_q \theta_r +a_{101}^{\alpha,\beta,\gamma} \theta_p \theta_r) u =0
\end{align*}
of order two,
where $a_{\kappa,\mu,\nu}^{\alpha,\beta,\gamma}$ are constants.
We can determine an equation which is independent of the equations (\ref{GKZpqr}) and satisfies $u=\mathcal{P}(p,q,r)$.
Although we need a much heavy calculation in order to determine it, we are able to achieve the goal.
Thus, we obtain the equation (\ref{NewDiffEq}).

In practice,
it is possible to check that  $u=\mathcal{P}(p,q,r)$ is a solution of (\ref{NewDiffEq}) directly.
There is an identity
\begin{align*}
&9(2n-1)(m+2n-2)(\ell +2m +3n -4) 
- (2\ell +4m +6n -9)^2 (2m+3n-3)\\
&- (2\ell +4m +6n -9)^2 (\ell-1)
+4 (\ell-1)(\ell + 2m + 3n -4) (\ell +4m +6n -8)\\
&+(m-1) (\ell +2 m +3n -4) (16m+30n -31) =0
\end{align*}
If we substitute $\mathcal{P}(p,q,r)$ for $u$ in the equation (\ref{NewDiffEq}), 
the left hand side of this identity appears as the coefficient of $p^\ell q^m r^n$.

(2) Because of Lemma \ref{LemKum},  the family $\pi$ of  (\ref{KummerFamily1}) attains all Kummer surfaces for principally polarized Abelian surfaces.
Therefore, as we saw in the beginning of this subsection,
the Picard-Fuchs system for the family $\pi$ must be given by a system of linear partial differential equation of rank five.
It is enough to see that the system of (\ref{GKZpqr}) and (\ref{NewDiffEq}) is of rank five.
Set $\varphi={}^t (u,\theta_p u, \theta_q u, \theta_r u, \theta_{p}^2 u) $.
From (\ref{GKZpqr}) and (\ref{NewDiffEq}), we obtain a Pfaffian 
$$
\Omega= M_p dp + M_q dq + M_r dr
$$
satisfying $d\varphi =\Omega \varphi$.
We omit the concrete expression of $\Omega$ from this proof, because it is much complicated (for a precise description, see Appendix).
By a direct calculation, we can check that
$$
\frac{\partial}{\partial p} M_q - \frac{\partial}{\partial q} M_p =[M_p, M_q] ,\quad
\frac{\partial}{\partial q} M_r - \frac{\partial}{\partial r} M_q =[M_q, M_r] , \quad
\frac{\partial}{\partial r} M_p - \frac{\partial}{\partial p} M_r =[M_r, M_p] 
$$
hold.
Therefore, $\Omega$ satisfies $d\Omega = \Omega \wedge \Omega$.
This means that the system of linear partial differential equations (\ref{GKZpqr}) and (\ref{NewDiffEq}) is of rank five.
\end{proof}

\begin{cor}\label{CorSingular}
The singular loci of 
the  Picard-Fuchs system (\ref{GKZpqr}) and (\ref{NewDiffEq}) 
is given by the union of divisors
\begin{align}\label{SingularLoci}
&\{p=0\} \cup \{q=0\}\cup \{r=0\} \notag \\
&\cup \{-q^2 + 2 p q^2 - p^2 q^2 + 4 q^3 - 4 r + 12 p r - 12 p^2 r +  4 p^3 r + 18 q r - 18 p q r + 27 r^2 =0 \} \notag \\
&\cup \{ -p^2 q^2 + 4 q^3 + 4 p^3 r - 18 p q r + 27 r^2 =0\}
\end{align}
in $\mathbb{C}^3 ={\rm Spec}(\mathbb{C}[p,q,r])$.
\end{cor}

\begin{proof}
The singular loci of the system   (\ref{GKZpqr}) and (\ref{NewDiffEq})  appear in the denominator of the Pfaffian $\Omega$ in the proof of Theorem \ref{ThmPFS}.
In practice, the divisors in (\ref{SingularLoci}) are coming from the explicit expression of $\Omega$ described in Appendix.
In particular, the last two divisors of (\ref{SingularLoci}) are corresponding to $d_2 $ and $d_3$ in  the notation of Appendix.

However, under the notation in Appendix, $\{d_1=0\}$ does not give a singular locus of the Picard-Fuchs system.
We can check it as follows.
We can calculate another explicit form of a Pfaffian,  if we change the basis $\varphi={}^t (u,\theta_p u, \theta_q u, \theta_r u, \theta_{p}^2 u) $ to another one, e.g. ${}^t (u,\theta_p u, \theta_q u, \theta_r u, \theta_{q}^2 u) $.
Then, the factor $d_1$ does not appear and another factor newly appear in the denominator of the new Pfaffian.  
\end{proof}

We are able to explain the reason why the last two divisors appear in the singular loci.
Recall that the elliptic fibration $(x,y,t) \mapsto t$ on the surface (\ref{KummerEq}) has  singular fibres of Kodaira type $I_4+6 I_2 + I_2^*$.
The discriminant of the right hand side of (\ref{KummerEq}) in $x$ is equal to
$
t^4 (t^3+p t^2  + q t +r  )^2 (t^3  - t^2 + p t^2+ q t   + r )^2
$
up to a constant factor.
Set $R_2(t)=(t^3  - t^2 + p t^2+ q t   + r ) $ and $R_3(t) =(t^3+p t^2  + q t +r  ) $.
Then, the discriminant of $R_2(t)$ ($R_3(t)$, resp.) is calculated as $d_2$ ($d_3$, resp.)  up to a constant factor.
On these loci,
two of singular fibres of Kodaira type $I_2$ collapse into a singular fibre of type $I_4$.

\section*{Appendix: Pfaffian $\Omega$ }

We give an explicit expression of the Pfaffian $\Omega = M_p dp + M_q dq + M_r dr$ in the proof of Theorem \ref{ThmPFS}.
It satisfies $d \varphi = \Omega \varphi$
for $\varphi= {}^t  (u,\theta_p u, \theta_q u, \theta_r u, \theta_{p}^2 u)$.
We remark that the proof of the main theorem (Theorem \ref{ThmPFS}) is ground on the concrete expression of $\Omega$.

We set
{\small
\begin{align*}
&d_1=-q^4 + 2 p q^4 - 4 q^2 r + 15 p q^2 r - 15 p^2 q^2 r + 6 q^3 r + 12 p r^2 - 36 p^2 r^2 + 24 p^3 r^2 - 81 r^3,\\
 &d_2= -q^2 + 2 p q^2 - p^2 q^2 + 4 q^3 - 4 r + 12 p r - 12 p^2 r + 
 4 p^3 r + 18 q r - 18 p q r + 27 r^2,\\
 & d_3=-p^2 q^2 + 4 q^3 + 4 p^3 r - 18 p q r + 27 r^2.
\end{align*}
}

\noindent
The entries of $M_p = (\tilde{p}_{j,k})$ are given as follows:
{\scriptsize
\begin{align*}
&\tilde{p}_{11}=\tilde{p}_{13}=\tilde{p}_{14}=\tilde{p}_{15}=\tilde{p}_{21}=\tilde{p}_{22}=\tilde{p}_{23}=\tilde{p}_{24}=0,\quad  \tilde{p}_{12}=\tilde{p}_{25}=1,\\
&\tilde{p}_{31}= p q (-q^3 + 4 (-1 + 2 p) q r + 36 r^2)/(8d_1),\\
&\tilde{p}_{32}=q (-2 (p^2 - 4 q) q^3 + q (8 p^2 (-1 + 2 p) + 9 (2 - 5 p) q) r + 
    18 (4 + 2 p (-5 + 4 p) - 15 q) r^2)/(4p d_1),\\
&\tilde{p}_{33} =p (-2 q^4 + 11 (-1 + 2 p) q^2 r + 3 (-4 (1 - 2 p)^2 + 27 q) r^2)/  (2 d_1)\\
&\tilde{p}_{34} =p q (-7 q^3 + 12 (-1 + 2 p) q r + 216 r^2)/ (4 d_1), \\
&\tilde{p}_{35}=q ((p - p^2 - 4 q) q^3 + (-1 + 2 p) q (4 (-1 + p) p + 9 q) r + 
   9 (-4 - 12 (-1 + p) p + 15 q) r^2)/(2p d_1), \\
 &\tilde{p}_{41}= p r^2 (-4 (-1 + p) p - 15 q)  /(4d_1),\\
 &\tilde{p}_{42} =r ((2 - 4 p) q^3 - 6 (4 + p (-5 + 4 p)) q r + 108 q^2 r + 
   r (-8 (-1 + p) p^3 - 27 (-2 + p) r))/(2 p d_1), \\
&\tilde{p}_{43} = p r (q^3 - 2 p q^3 + 4 q r - 7 p q r + 7 p^2 q r - 36 q^2 r - 9 r^2 + 18 p r^2)/(q d_1),\\
&\tilde{p}_{44}= p r ((-4 + 8 p) q^2 - 24 (-1 + p) p r - 81 q r)/(2 d_1),\\
&\tilde{p}_{45}=r ((-1 + 2 p) q^3 + 6 (2 + 3 (-1 + p) p) q r - 54 q^2 r + 
   r (-4 (-1 + p)^2 p^2 + 27 (-1 + 2 p) r))/(p d_1),\\
&\tilde{p}_{51}=  -p^3 (2 q^8 - 2 p q^8 + 2 p^2 q^8 - 24 q^9 + 16 q^6 r - 49 p q^6 r + 
   51 p^2 q^6 r - 34 p^3 q^6 r - 190 q^7 r + 380 p q^7 r + 
   32 q^4 r^2 - 232 p q^4 r^2  \\
 &  + 456 p^2 q^4 r^2- 448 p^3 q^4 r^2 + 224 p^4 q^4 r^2 - 277 q^5 r^2 + 2124 p q^5 r^2 - 2124 p^2 q^5 r^2 - 966 q^6 r^2 - 272 p q^2 r^3 +  1104 p^2 q^2 r^3 - 1696 p^3 q^2 r^3   \\
&+ 1440 p^4 q^2 r^3 - 576 p^5 q^2 r^3 + 504 q^3 r^3 + 1320 p q^3 r^3 - 6984 p^2 q^3 r^3 + 4656 p^3 q^3 r^3 - 4239 q^4 r^3 + 8478 p q^4 r^3 + 384 p^2 r^4 - 1152 p^3 r^4    \\
 & + 1536 p^4 r^4  -  1152 p^5 r^4  + 384 p^6 r^4 + 432 q r^4 - 2592 p q r^4+  5184 p^3 q r^4 - 2592 p^4 q r^4 - 4968 q^2 r^4  + 24300 p q^2 r^4 - 24300 p^2 q^2 r^4   \\
 & + 3240 q^3 r^4 - 3888 r^5 + 18144 p r^5 - 31104 p^2 r^5 + 20736 p^3 r^5 + 8748 q r^5 - 17496 p q r^5 + 43740 r^6) /(8 d_1 d_2 d_3),\\
&\tilde{p}_{52}=(-5 p^3 q^8 + 7 p^4 q^8 - 6 p^5 q^8 - 32 q^9 + 128 p q^9 - 108 p^2 q^9 + 56 p^3 q^9 + 128 q^10 - 448 p q^10 - 40 p^3 q^6 r + 137 p^4 q^6 r - 164 p^5 q^6 r \\
&+ 99 p^6 q^6 r - 256 q^7 r +1484 p q^7 r - 2793 p^2 q^7 r + 2112 p^3 q^7 r - 904 p^4 q^7 r +  1280 q^8 r - 5760 p q^8 r + 6912 p^2 q^8 r - 768 q^9 r  -  80 p^3 q^4 r^2 \\ 
& + 600 p^4 q^4 r^2 - 1296 p^5 q^4 r^2+ 1336 p^6 q^4 r^2 - 624 p^7 q^4 r^2 - 512 q^5 r^2 + 4704 p q^5 r^2 -  15266 p^2 q^5 r^2 + 20762 p^3 q^5 r^2  - 14202 p^4 q^5 r^2  \\
&+  5076 p^5 q^5 r^2 + 2856 q^6 r^2 - 19536 p q^6 r^2 +  50853 p^2 q^6 r^2 - 36546 p^3 q^6 r^2 - 1728 q^7 r^2 - 4320 p q^7 r^2 + 656 p^4 q^2 r^3 - 2864 p^5 q^2 r^3   \\
&+  4800 p^6 q^2 r^3  - 4176 p^7 q^2 r^3 + 1584 p^8 q^2 r^3 +  3264 p q^3 r^3 - 22752 p^2 q^3 r^3 + 55872 p^3 q^3 r^3 -  63216 p^4 q^3 r^3 + 38976 p^5 q^3 r^3  \\
&- 11280 p^6 q^3 r^3 - 1728 q^4 r^3 - 5778 p q^4 r^3 + 74115 p^2 q^4 r^3 -154008 p^3 q^4 r^3 + 77868 p^4 q^4 r^3 + 9504 q^5 r^3 -  56376 p q^5 r^3 + 84564 p^2 q^5 r^3 \\
& - 960 p^5 r^4 + 3264 p^6 r^4 -  4800 p^7 r^4+ 3648 p^8 r^4 - 1152 p^9 r^4 - 6048 p^2 q r^4 +  35424 p^3 q r^4 - 69984 p^4 q r^4 + 67392 p^5 q r^4 -  34560 p^6 q r^4 \\
&  + 6912 p^7 q r^4 - 3456 q^2 r^4 + 29808 p q^2 r^4-  61128 p^2 q^2 r^4 - 30564 p^3 q^2 r^4 + 117612 p^4 q^2 r^4 -  44712 p^5 q^2 r^4 + 10368 q^3 r^4  \\
 & - 139320 p q^3 r^4 + 414720 p^2 q^3 r^4 - 324000 p^3 q^3 r^4 + 29160 q^4 r^4 -  32076 p q^4 r^4 + 18144 p r^5 - 104976 p^2 r^5 + 221616 p^3 r^5 -  225504 p^4 r^5    \\
&+ 119232 p^5 r^5 - 32400 p^6 r^5 - 69984 p q r^5 +  472392 p^2 q r^5 - 816480 p^3 q r^5 + 443232 p^4 q r^5 + 
 5832 q^2 r^5 - 196830 p q^2 r^5 + 137781 p^2 q^2 r^5 \\
 &+  104976 q^3 r^5 - 69984 r^6 - 69984 p r^6 + 559872 p^2 r^6 - 489888 p^3 r^6 + 314928 q r^6 - 157464 p q r^6 + 472392 r^7)
/(4d_1 d_2 d_3),\\
&\tilde{p}_{53}=-p^3 (3 q^9 - 48 q^{10} + 22 q^7 r - 41 p q^7 r - 9 p^2 q^7 r +  6 p^3 q^7 r - 364 q^8 r + 728 p q^8 r + 32 q^5 r^2 -  219 p q^5 r^2 + 179 p^2 q^5 r^2 + 80 p^3 q^5 r^2 \\
& - 40 p^4 q^5 r^2 - 451 q^6 r^2 + 3984 p q^6 r^2 - 3984 p^2 q^6 r^2 -  2208 q^7 r^2 - 32 q^3 r^3 - 232 p q^3 r^3 + 864 p^2 q^3 r^3 - 496 p^3 q^3 r^3 - 120 p^4 q^3 r^3  \\
& + 48 p^5 q^3 r^3 + 1209 q^4 r^3 + 2358 p q^4 r^3 - 14328 p^2 q^4 r^3 +   9552 p^3 q^4 r^3 - 10395 q^5 r^3 + 20790 p q^5 r^3 - 48 p q r^4 +  912 p^2 q r^4- 1824 p^3 q r^4   \\
&+ 1152 p^4 q r^4 - 288 p^5 q r^4 + 96 p^6 q r^4 + 1368 q^2 r^4 - 7128 p q^2 r^4 - 648 p^2 q^2 r^4 + 
   15552 p^3 q^2 r^4 - 7776 p^4 q^2 r^4 - 14472 q^3 r^4 + 
   72414 p q^3 r^4    \\
& - 72414 p^2 q^3 r^4+ 2268 q^4 r^4 + 1296 r^5 -  8640 p r^5 + 20736 p^2 r^5 - 22464 p^3 r^5 + 12960 p^4 r^5 - 
   5184 p^5 r^5 - 14256 q r^5 + 82620 p q r^5  \\
& - 162324 p^2 q r^5 + 108216 p^3 q r^5 + 4374 q^2 r^5 - 8748 p q^2 r^5 - 37908 r^6 + 139968 p r^6 - 139968 p^2 r^6 + 135594 q r^6)/(2q d_1 d_2 d_3),\\
&\tilde{p}_{54}=p^3 (-10 q^8 - 2 p q^8 + 2 p^2 q^8 + 168 q^9 - 80 q^6 r + 163 p q^6 r - 9 p^2 q^6 r + 6 p^3 q^6 r + 1266 q^7 r - 2532 p q^7 r - 160 q^4 r^2 + 984 p q^4 r^2  \\
& - 1104 p^2 q^4 r^2 + 240 p^3 q^4 r^2 - 120 p^4 q^4 r^2 + 1671 q^5 r^2 - 13140 p q^5 r^2 + 13140 p^2 q^5 r^2 + 5562 q^6 r^2 + 1200 p q^2 r^3 - 3744 p^2 q^2 r^3  \\
& + 2976 p^3 q^2 r^3  -720 p^4 q^2 r^3 + 288 p^5 q^2 r^3 - 2376 q^3 r^3 - 9720 p q^3 r^3 + 43416 p^2 q^3 r^3 - 28944 p^3 q^3 r^3 +    21789 q^4 r^3 - 43578 p q^4 r^3   \\
 & - 1728 p^2 r^4 + 3456 p^3 r^4- 1728 p^4 r^4 - 1296 q r^4 + 7776 p q r^4 + 15552 p^2 q r^4 - 
   46656 p^3 q r^4 + 23328 p^4 q r^4 + 26892 q^2 r^4 -  121500 p q^2 r^4    \\ 
 & + 121500 p^2 q^2 r^4 - 29160 q^3 r^4+ 23328 r^5  - 116640 p r^5 + 209952 p^2 r^5 - 139968 p^3 r^5 - 78732 q r^5 + 157464 p q r^5 - 236196 r^6)/(4 d_1 d_2 d_3),\\   
&\tilde{p}_{55}= -3 (-2 p^4 q^8 + 2 p^5 q^8 - 8 q^9 + 32 p q^9 - 32 p^2 q^9 + 16 p^3 q^9 + 32 q^{10} - 96 p q^{10} + p^3 q^6 r - 18 p^4 q^6 r +    49 p^5 q^6 r - 32 p^6 q^6 r - 64 q^7 r  \\
& + 374 p q^7 r -744 p^2 q^7 r + 604 p^3 q^7 r - 228 p^4 q^7 r + 320 q^8 r - 1344 p q^8 r + 1536 p^2 q^8 r - 192 q^9 r + 8 p^3 q^4 r^2 -  64 p^4 q^4 r^2 + 264 p^5 q^4 r^2  \\
& - 400 p^6 q^4 r^2 +  192 p^7 q^4 r^2-128 q^5 r^2 + 1200 p q^5 r^2 - 3996 p^2 q^5 r^2 + 5716 p^3 q^5 r^2 - 3876 p^4 q^5 r^2 + 1164 p^5 q^5 r^2 + 714 q^6 r^2 - 4896 p q^6 r^2   \\
& + 12024 p^2 q^6 r^2 - 8532 p^3 q^6 r^2- 432 q^7 r^2 - 432 p q^7 r^2 + 16 p^3 q^2 r^3 - 96 p^4 q^2 r^3 +    432 p^5 q^2 r^3 - 1120 p^6 q^2 r^3 + 1248 p^7 q^2 r^3 - 480 p^8 q^2 r^3 \\
&+ 864 p q^3 r^3 - 6048 p^2 q^3 r^3 +15504 p^3 q^3 r^3  - 18432 p^4 q^3 r^3  + 11088 p^5 q^3 r^3 - 2880 p^6 q^3 r^3 - 432 q^4 r^3 - 1917 p q^4 r^3 +   19008 p^2 q^4 r^3  \\
& - 38070 p^3 q^4 r^3 + 20034 p^4 q^4 r^3 + 2376 q^5 r^3 - 11340 p q^5 r^3 + 16848 p^2 q^5 r^3 - 384 p^6 r^4 
   +  1152 p^7 r^4 - 1152 p^8 r^4 + 384 p^9 r^4 - 1728 p^2 q r^4  \\
 &  + 10368 p^3 q r^4- 22464 p^4 q r^4 + 24192 p^5 q r^4 - 13824 p^6 q r^4 + 3456 p^7 q r^4 - 864 q^2 r^4 + 7128 p q^2 r^4 - 13608 p^2 q^2 r^4 - 9288 p^3 q^2 r^4    \\
&  + 32400 p^4 q^2 r^4 -14256 p^5 q^2 r^4 + 2592 q^3 r^4 - 30132 p q^3 r^4+ 89424 p^2 q^3 r^4 - 73872 p^3 q^3 r^4 + 7290 q^4 r^4 -10206 p q^4 r^4 + 3888 p r^5 \\
& - 23328 p^2 r^5 + 53136 p^3 r^5 -  62208 p^4 r^5 + 42768 p^5 r^5 - 15552 p^6 r^5-11664 p q r^5 + 93312 p^2 q r^5 - 186624 p^3 q r^5 + 116640 p^4 q r^5   \\
& + 1458 q^2 r^5 - 54675 p q^2 r^5 + 43740 p^2 q^2 r^5 +  26244 q^3 r^5 - 17496 r^6 + 104976 p^2 r^6 - 104976 p^3 r^6 +  78732 q r^6 - 78732 p q r^6\\
& + 118098 r^7) /(2 d_1 d_2 d_3).
\end{align*}
}

\noindent
The entries of $M_q = (\tilde{q}_{j,k})$ are given as  follows:
{\scriptsize
\begin{align*}
&\tilde{q}_{11}=\tilde{q}_{12}=\tilde{q}_{14}=\tilde{q}_{15} =0,  \quad \tilde{q}_{13}=1,\\
&\tilde{q}_{21}=p q (-q^3 + 4 (-1 + 2 p) q r + 36 r^2) /(8d_1),\\
&\tilde{q}_{22}=q (-2 (p^2 - 4 q) q^3 + q (8 p^2 (-1 + 2 p) + 9 (2 - 5 p) q) r + 
    18 (4 + 2 p (-5 + 4 p) - 15 q) r^2)/(4pd_1),\\
&\tilde{q}_{23}=p (-2 q^4 + 11 (-1 + 2 p) q^2 r + 3 (-4 (1 - 2 p)^2 + 27 q) r^2)/(2 d_1),\\
&\tilde{q}_{24}=p q (-7 q^3 + 12 (-1 + 2 p) q r + 216 r^2)/(4d_1),\\
&\tilde{q}_{25}=q ((p - p^2 - 4 q) q^3 + (-1 + 2 p) q (4 (-1 + p) p + 9 q) r + 9 (-4 - 12 (-1 + p) p + 15 q) r^2)/(2p d_1),\\
&\tilde{q}_{31}= q^2 r (-4 (-1 + p) p - 15 q) /(4d_1),\\
&\tilde{q}_{32}=q^2 ((2 - 4 p) q^3 - 6 (4 + p (-5 + 4 p)) q r + 108 q^2 r + r (-8 (-1 + p) p^3 - 27 (-2 + p) r))/(2p^2 d_1),\\
&\tilde{q}_{33}= -r (8 (-1 + p) p q^2 + 30 q^3 - 3 (-1 + 2 p) (4 (-1 + p) p + 3 q) r + 81 r^2)/d_1,\\
&\tilde{q}_{34}=(4 (-1 + 2 p) q^4 - 3 q^2 (8 (-1 + p) p + 27 q) r )/(2d_1),\\
&\tilde{q}_{35}=q^2 ((-1 + 2 p) q^3 + 6 (2 + 3 (-1 + p) p) q r - 54 q^2 r +  r (-4 (-1 + p)^2 p^2 + 27 (-1 + 2 p) r))/(p^2 d_1),\\
&\tilde{q}_{41}=q^2  r (4 (-1 + p) p + 15 q)/(8d_1),\\
&\tilde{q}_{42}=q r (2 q (4 (-1 + p) p^3 + (16 + 3 p (-10 + 9 p)) q - 60 q^2) - 3 (8 (-1 + p) p (-1 + 2 p) - 9 (-2 + p) q) r + 162 r^2)/(4p^2 d_1),\\
&\tilde{q}_{43}=r (8 (-1 + p) p q^2 + 30 q^3 - 3 (-1 + 2 p) (4 (-1 + p) p + 3 q) r + 81 r^2)/(2d_1),\\
&\tilde{q}_{44}=q^2 ((4 - 8 p) q^2 + 24 (-1 + p) p r + 81 q r)/(4d_1),\\
&\tilde{q}_{45}=q r (q (4 (-1 + p)^2 p^2 + (-16 - 33 (-1 + p) p) q + 60 q^2) + 3 (-1 + 2 p) (4 (-1 + p) p - 9 q) r - 81 r^2)/(2p^2 d_1),\\
&\tilde{q}_{51}=-p q (p^3 q^7 - p^4 q^7 - 4 q^8 + 8 p^2 q^8 + 16 q^9 + 10 p^3 q^5 r - 28 p^4 q^5 r + 18 p^5 q^5 r - 32 q^6 r + 57 p q^6 r + 74 p^2 q^6 r - 140 p^3 q^6 r + 136 q^7 r  \\
&- 200 p q^7 r +  48 p^3 q^3 r^2 - 224 p^4 q^3 r^2 + 320 p^5 q^3 r^2 - 144 p^6 q^3 r^2 - 64 q^4 r^2 + 264 p q^4 r^2 + 64 p^2 q^4 r^2 - 1108 p^3 q^4 r^2 + 980 p^4 q^4 r^2 + 405 q^5 r^2 \\
& - 1449 p q^5 r^2 + 1188 p^2 q^5 r^2 - 360 q^6 r^2 + 96 p^3 q r^3 - 480 p^4 q r^3 + 1024 p^5 q r^3 - 992 p^6 q r^3 + 352 p^7 q r^3 +  144 p q^2 r^3 - 432 p^2 q^2 r^3 \\
& - 1728 p^3 q^2 r^3+  4464 p^4 q^2 r^3 - 2592 p^5 q^2 r^3 + 360 q^3 r^3 - 828 p q^3 r^3 + 3942 p^2 q^3 r^3 - 1746 p^3 q^3 r^3 - 1242 q^4 r^3 - 4968 p q^4 r^3 - 576 p^2 r^4 \\
 &+ 1440 p^3 r^4  + 576 p^4 r^4 - 2592 p^5 r^4 + 1152 p^6 r^4 - 432 q r^4 + 2592 p q r^4 - 4752 p^2 q r^4 + 864 p^3 q r^4 - 1728 p^4 q r^4 +  2916 q^2 r^4- 15066 p q^2 r^4 \\
 &  + 32076 p^2 q^2 r^4 - 
   7047 q^3 r^4 + 3888 r^5 - 7776 p r^5 + 17496 p^2 r^5 - 
   13608 p^3 r^5 - 14580 q r^5 - 14580 p q r^5 - 26244 r^6)/(8d_1 d_2 d_3),\\
&\tilde{q}_{52}= q (-2 p^5 q^7 + 2 p^6 q^7 + 16 p q^8 - 46 p^2 q^8 + 56 p^3 q^8 -  32 p^4 q^8 - 32 q^9 - 96 p q^9 + 224 p^2 q^9 + 128 q^{10} -  20 p^5 q^5 r + 56 p^6 q^5 r - 36 p^7 q^5 r \\
&+ 128 p q^6 r  -   600 p^2 q^6 r + 1140 p^3 q^6 r - 1135 p^4 q^6 r + 533 p^5 q^6 r -  200 q^7 r - 556 p q^7 r + 3496 p^2 q^7 r - 3640 p^3 q^7 r +    864 q^8 r - 720 p q^8 r   \\
& - 96 p^5 q^3 r^2 + 448 p^6 q^3 r^2 -  640 p^7 q^3 r^2 + 288 p^8 q^3 r^2 + 256 p q^4 r^2 -   1952 p^2 q^4 r^2 + 5692 p^3 q^4 r^2 - 8808 p^4 q^4 r^2 +   7820 p^5 q^4 r^2 - 3248 p^6 q^4 r^2 \\
& - 576 q^5 r^2 +  1098 p q^5 r^2 + 8829 p^2 q^5 r^2 - 26334 p^3 q^5 r^2 +18810 p^4 q^5 r^2 + 3312 q^6 r^2 - 11304 p q^6 r^2 + 9216 p^2 q^6 r^2 - 2592 q^7 r^2  - 192 p^5 q r^3 \\
& + 960 p^6 q r^3 - 2048 p^7 q r^3 + 1984 p^8 q r^3 - 704 p^9 q r^3 -1152 p^2 q^2 r^3 + 8064 p^3 q^2 r^3 - 20880 p^4 q^2 r^3 +  29376 p^5 q^2 r^3 - 23040 p^6 q^2 r^3 \\
& + 7920 p^7 q^2 r^3 - 1152 q^3 r^3 + 5904 p q^3 r^3 - 4176 p^2 q^3 r^3 - 42660 p^3 q^3 r^3 + 85068 p^4 q^3 r^3 - 45504 p^5 q^3 r^3 + 8154 q^4 r^3 - 24273 p q^4 r^3 \\
& + 54729 p^2 q^4 r^3 - 27108 p^3 q^4 r^3 - 11664 q^5 r^3 - 33048 p q^5 r^3 + 2880 p^3 r^4 - 15552 p^4 r^4 + 33408 p^5 r^4 - 37440 p^6 r^4 + 21312 p^7 r^4   \\
& - 4608 p^8 r^4 - 864 p q r^4 + 5616 p^2 q r^4 - 42336 p^3 q r^4 + 128952 p^4 q r^4 - 140400 p^5 q r^4 + 52920 p^6 q r^4 - 3888 q^2 r^4 + 31104 p q^2 r^4  \\
&- 61236 p^2 q^2 r^4 + 23814 p^3 q^2 r^4- 28674 p^4 q^2 r^4 + 31590 q^3 r^4 - 178848 p q^3 r^4 + 303750 p^2 q^3 r^4 - 
52488 q^4 r^4 - 7776 r^5 + 42768 p r^5\\
& - 89424 p^2 r^5 +  29160 p^3 r^5 + 71928 p^4 r^5 - 50544 p^5 r^5 + 64152 q r^5 -  177876 p q r^5 + 392202 p^2 q r^5 - 204120 p^3 q r^5 -  118098 q^2 r^5 \\
& - 190269 p q^2 r^5 + 52488 r^6 - 367416 p r^6 +  459270 p^2 r^6 - 196830 q r^6) /(4p d_1 d_2 d_3), \\
&\tilde{q}_{53} =p (-p^2 q^8 - p^3 q^8 + 2 p^4 q^8 + 16 q^9 + 4 p q^9 - 40 p^2 q^9 -  64 q^{10} - 16 p^2 q^6 r + 33 p^3 q^6 r + 4 p^4 q^6 r -  21 p^5 q^6 r + 152 q^7 r - 295 p q^7 r  \\
& - 258 p^2 q^7 r + 570 p^3 q^7 r - 640 q^8 r + 968 p q^8 r - 80 p^2 q^4 r^2 +  292 p^3 q^4 r^2 - 192 p^4 q^4 r^2 - 204 p^5 q^4 r^2 + 184 p^6 q^4 r^2 + 448 q^5 r^2 - 2054 p q^5 r^2   \\
& + 1689 p^2 q^5 r^2+ 2695 p^3 q^5 r^2 - 3296 p^4 q^5 r^2 -  2508 q^6 r^2 + 8919 p q^6 r^2 - 7590 p^2 q^6 r^2 + 1728 q^7 r^2 -  128 p^2 q^2 r^3 + 528 p^3 q^2 r^3 - 560 p^4 q^2 r^3    \\
&- 864 p^5 q^2 r^3 + 1904 p^6 q^2 r^3 - 880 p^7 q^2 r^3 +  384 q^3 r^3 - 3168 p q^3 r^3 + 8112 p^2 q^3 r^3 +  240 p^3 q^3 r^3 - 15912 p^4 q^3 r^3 + 11496 p^5 q^3 r^3 - 3294 q^4 r^3\\
& + 13356 p q^4 r^3 - 32832 p^2 q^4 r^3 +  14904 p^3 q^4 r^3 + 4968 q^5 r^3 + 21006 p q^5 r^3 - 192 p^3 r^4 +  1152 p^4 r^4 - 3264 p^5 r^4 + 5376 p^6 r^4 - 4608 p^7 r^4   \\
& + 1536 p^8 r^4+ 288 p q r^4 + 432 p^2 q r^4 + 4320 p^3 q r^4 -  29952 p^4 q r^4 + 43632 p^5 q r^4 - 20736 p^6 q r^4 + 3024 q^2 r^4 - 18360 p q^2 r^4 + 31104 p^2 q^2 r^4  \\
& -11880 p^3 q^2 r^4 + 21168 p^4 q^2 r^4 - 20250 q^3 r^4 +  99387 p q^3 r^4 - 179010 p^2 q^3 r^4 + 32076 q^4 r^4 + 2592 r^5  -  18144 p r^5 + 40176 p^2 r^5  \\
& - 16848 p^3 r^5 - 31104 p^4 r^5 +  20736 p^5 r^5 - 29160 q r^5 + 83592 p q r^5 - 169128 p^2 q r^5 + 104976 p^3 q r^5 + 62694 q^2 r^5 + 143613 p q^2 r^5 \\
 & - 17496 r^6 + 157464 p r^6 - 244944 p^2 r^6+ 118098 q r^6) /(4 d_1 d_2 d_3),\\
& \tilde{q}_{54}= -p q (2 p^2 q^7 + p^3 q^7 - 3 p^4 q^7 - 28 q^8 - 8 p q^8 +72 p^2 q^8 + 112 q^9 + 56 p^3 q^5 r - 152 p^4 q^5 r +96 p^5 q^5 r - 160 q^6 r + 101 p q^6 r + 1050 p^2 q^6 r \\
&-1384 p^3 q^6 r + 696 q^7 r - 840 p q^7 r - 32 p^2 q^3 r^2 + 272 p^3 q^3 r^2 - 928 p^4 q^3 r^2 + 1248 p^5 q^3 r^2 -  560 p^6 q^3 r^2 - 192 q^4 r^2 + 840 p q^4 r^2 \\
& + 2142 p^2 q^4 r^2 - 9282 p^3 q^4 r^2 + 7140 p^4 q^4 r^2 + 1539 q^5 r^2 - 6885 p q^5 r^2 + 6696 p^2 q^5 r^2 - 1944 q^6 r^2 + 384 p^3 q r^3 -1536 p^4 q r^3 + 2880 p^5 q r^3  \\
 &- 2688 p^6 q r^3+ 960 p^7 q r^3 +   720 p q^2 r^3 - 2160 p^2 q^2 r^3 - 10800 p^3 q^2 r^3 + 26784 p^4 q^2 r^3 - 14688 p^5 q^2 r^3 + 2376 q^3 r^3 - 5616 p q^3 r^3 \\
& + 26082 p^2 q^3 r^3- 16254 p^3 q^3 r^3 - 9558 q^4 r^3 - 24948 p q^4 r^3 - 2592 p^2 r^4 + 3456 p^3 r^4 +  15552 p^4 r^4 - 28512 p^5 r^4 + 12096 p^6 r^4 - 1296 q r^4 \\
&+ 9072 p q r^4 - 21384 p^2 q r^4 + 8424 p^3 q r^4 -  14256 p^4 q r^4 + 11664 q^2 r^4 - 72900 p q^2 r^4 +  180792 p^2 q^2 r^4 - 41553 q^3 r^4 + 23328 r^5 - 52488 p r^5  \\
& + 122472 p^2 r^5- 99144 p^3 r^5 - 96228 q r^5 - 21870 p q r^5 - 157464 r^6)/(4 d_1 d_2 d_3),\\
&\tilde{q}_{55}=-q (-p^4 q^7 + 2 p^5 q^7 - p^6 q^7 + 12 p q^8 - 40 p^2 q^8 +    56 p^3 q^8 - 36 p^4 q^8 - 16 q^9 - 64 p q^9 + 144 p^2 q^9 +    64 q^{10} - 10 p^4 q^5 r  + 38 p^5 q^5 r  \\
 & - 46 p^6 q^5 r  + 18 p^7 q^5 r + 96 p q^6 r - 493 p^2 q^6 r + 1024 p^3 q^6 r -  1085 p^4 q^6 r + 522 p^5 q^6 r - 100 q^7 r - 432 p q^7 r + 2328 p^2 q^7 r - 2400 p^3 q^7 r \\
 & + 432 q^8 r - 288 p q^8 r -  48 p^4 q^3 r^2 + 272 p^5 q^3 r^2 - 544 p^6 q^3 r^2 +    464 p^7 q^3 r^2 - 144 p^8 q^3 r^2 + 192 p q^4 r^2  -  1512 p^2 q^4 r^2 + 4534 p^3 q^4 r^2   \\
& - 7020 p^4 q^4 r^2   + 6006 p^5 q^4 r^2 - 2328 p^6 q^4 r^2 - 288 q^5 r^2 +  288 p q^5 r^2 + 6246 p^2 q^5 r^2 - 16956 p^3 q^5 r^2 + 11763 p^4 q^5 r^2 + 1656 q^6 r^2   \\
& - 5832 p q^6 r^2 + 4968 p^2 q^6 r^2 - 1296 q^7 r^2 - 96 p^4 q r^3 + 576 p^5 q r^3 - 1504 p^6 q r^3 + 2016 p^7 q r^3 -1344 p^8 q r^3 + 352 p^9 q r^3 - 720 p^2 q^2 r^3  \\
& + 5328 p^3 q^2 r^3 - 14400 p^4 q^2 r^3  + 20448 p^5 q^2 r^3 -15984 p^6 q^2 r^3 + 5328 p^7 q^2 r^3 - 576 q^3 r^3 +  3456 p q^3 r^3 - 3240 p^2 q^3 r^3 - 24012 p^3 q^3 r^3  \\
& +  51408 p^4 q^3 r^3  - 27972 p^5 q^3 r^3 + 4077 q^4 r^3 - 14904 p q^4 r^3 + 34830 p^2 q^4 r^3 - 19116 p^3 q^4 r^3 - 5832 q^5 r^3 - 15552 p q^5 r^3  + 2016 p^3 r^4   \\
& - 12096 p^4 r^4 + 30240 p^5 r^4 - 40896 p^6 r^4+ 29376 p^7 r^4 - 8640 p^8 r^4  -  19008 p^3 q r^4 + 80352 p^4 q r^4 - 103680 p^5 q r^4 +  45360 p^6 q r^4 - 1944 q^2 r^4 \\
 &  + 13608 p q^2 r^4 -  25272 p^2 q^2 r^4  + 5832 p^3 q^2 r^4   - 21384 p^4 q^2 r^4 +  15795 q^3 r^4 - 91854 p q^3 r^4 + 178605 p^2 q^3 r^4 -  26244 q^4 r^4 - 3888 r^5   \\
 &+ 23328 p r^5 - 58320 p^2 r^5 + 42768 p^3 r^5 + 11664 p^4 r^5 - 11664 p^5 r^5 + 32076 q r^5 -  100602 p q r^5 + 266814 p^2 q r^5 - 196830 p^3 q r^5 - 59049 q^2 r^5 \\
& - 91854 p q^2 r^5 + 26244 r^6 - 196830 p r^6  +  314928 p^2 r^6   - 98415 q r^6)   /(2p d_1 d_2 d_3).
\end{align*}
}

\noindent
The entries of $M_r = (\tilde{r}_{j,k})$ are given as  follows:
{\scriptsize
\begin{align*}
&\tilde{r}_{11}=\tilde{r}_{12}=\tilde{r}_{13}=\tilde{r}_{15}=0, \quad  \tilde{r}_{14}=1,\\
&\tilde{r}_{21}=-p r^2 (-4 p + 4 p^2 + 15 q) /(4d_1 ), \\
&\tilde{r}_{22}=r ((2 - 4 p) q^3 - 6 (4 + p (-5 + 4 p)) q r + 108 q^2 r + r (-8 (-1 + p) p^3 - 27 (-2 + p) r))/(2p d_1),\\
&\tilde{r}_{23}=p r ((1 - 2 p) q^3 + (4 + 7 (-1 + p) p - 36 q) q r + 9 (-1 + 2 p) r^2)/(q d_1),\\
&\tilde{r}_{24}=p r ((-4 + 8 p) q^2 - 24 (-1 + p) p r - 81 q r)/(2d_1),\\
&\tilde{r}_{25}=r ((-1 + 2 p) q^3 + 6 (2 + 3 (-1 + p) p) q r - 54 q^2 r + r (-4 (-1 + p)^2 p^2 + 27 (-1 + 2 p) r))/(p d_1),\\
&\tilde{r}_{31}=q^2 r (4 (-1 + p) p + 15 q) /(8d_1),\\
&\tilde{r}_{32}=q r (2 q (4 (-1 + p) p^3 + (16 + 3 p (-10 + 9 p)) q - 60 q^2) - 3 (8 (-1 + p) p (-1 + 2 p) - 9 (-2 + p) q) r + 162 r^2)/(4p^2 d_1),\\
&\tilde{r}_{33}=r (8 (-1 + p) p q^2 + 30 q^3 - 3 (-1 + 2 p) (4 (-1 + p) p + 3 q) r +  81 r^2)/(2d_1),\\
&\tilde{r}_{34}= q^2 ((4 - 8 p) q^2 + 24 (-1 + p) p r + 81 q r)/(4d_1),\\
&\tilde{r}_{35}= q r (q (4 (-1 + p)^2 p^2 + (-16 - 33 (-1 + p) p) q + 60 q^2) +  3 (-1 + 2 p) (4 (-1 + p) p - 9 q) r - 81 r^2)/(2p^2 d_1),\\
&\tilde{r}_{41}= r ((p - p^2 - 4 q) q^2 + (-1 + 2 p) q r + 9 r^2)/(4d_1),\\
&\tilde{r}_{42}= r (q^2 (-2 (-1 + p) p^3 + (-8 + (15 - 14 p) p) q + 32 q^2) +  2 q (p (3 + 2 p (-5 + 4 p) - 9 q) + 9 q) r +  9 (2 + p (-5 + 4 p) - 12 q) r^2)/(2p^2 d_1),\\
&\tilde{r}_{43}= r (-2 q^3 ((-1 + p) p + 4 q) + (-1 + 2 p) q (3 (-1 + p) p + 5 q) r -  3 (1 - 2 p)^2 r^2)/(q d_1),\\
&\tilde{r}_{44}=((-1 + 2 p) q^4 - 2 q^2 (3 (-1 + p) p + 11 q) r + 3 (-1 + 2 p) q r^2 + 54 r^3)/(2 d_1),\\
&\tilde{r}_{45}=r (-q^2 (-(-1 + p)^2 + 4 q) (-p^2 + 4 q) + (-1 + 2 p) q (-2 (-1 + p) p + 9 q) r +  9 (-1 - 3 (-1 + p) p + 6 q) r^2)/(p^2 d_1),\\
&\tilde{r}_{51}=p r (p^3 q^6 - 2 p^4 q^6 + p^5 q^6 - 6 p q^7 + 10 p^2 q^7 - 10 p^3 q^7 + 40 p q^8 + 16 p^3 q^4 r - 56 p^4 q^4 r +  64 p^5 q^4 r - 24 p^6 q^4 r - 80 p q^5 r + 243 p^2 q^5 r \\
&-  311 p^3 q^5 r + 196 p^4 q^5 r + 120 q^6 r + 198 p q^6 r -  508 p^2 q^6 r - 480 q^7 r + 48 p^3 q^2 r^2 - 176 p^4 q^2 r^2 +   256 p^5 q^2 r^2 - 176 p^6 q^2 r^2 + 48 p^7 q^2 r^2  \\
& - 224 p q^3 r^2 + 832 p^2 q^3 r^2 - 1632 p^3 q^3 r^2 + 1704 p^4 q^3 r^2 - 776 p^5 q^3 r^2 + 480 q^4 r^2 + 564 p q^4 r^2 -  3294 p^2 q^4 r^2 + 3738 p^3 q^4 r^2 - 2160 q^5 r^2 \\
&-  594 p q^5 r^2 - 128 p^5 r^3 + 384 p^6 r^3 - 384 p^7 r^3 + 
   128 p^8 r^3 + 336 p^2 q r^3 - 960 p^3 q r^3 + 1536 p^4 q r^3 - 
   1296 p^5 q r^3 + 384 p^6 q r^3 + 216 p q^2 r^3 \\
&-  6696 p^2 q^2 r^3 + 13824 p^3 q^2 r^3 - 9936 p^4 q^2 r^3 +  810 q^3 r^3 - 1701 p q^3 r^3 + 13230 p^2 q^3 r^3 - 6480 q^4 r^3 -  864 p r^4 + 432 p^2 r^4 + 5184 p^3 r^4\\
& - 10368 p^4 r^4 +    5616 p^5 r^4 + 3240 q r^4 - 1296 p q r^4 + 5832 p^2 q r^4 - 6480 p^3 q r^4 - 14580 q^2 r^4 + 729 p q^2 r^4 + 5832 p^2 r^5 -   21870 q r^5)/(8 d_1 d_2 d_3),\\
 &\tilde{r}_{52}=  -r (-2 p^5 q^6 + 4 p^6 q^6 - 2 p^7 q^6 + 4 p^2 q^7 + 5 p^3 q^7 - 19 p^4 q^7 + 22 p^5 q^7 + 16 q^8 - 96 p q^8 + 76 p^2 q^8 -  88 p^3 q^8 - 64 q^9 + 448 p q^9 - 32 p^5 q^4 r \\
 &+ 112 p^6 q^4 r - 128 p^7 q^4 r + 48 p^8 q^4 r + 96 p^2 q^5 r - 234 p^3 q^5 r +   86 p^4 q^5 r + 284 p^5 q^5 r - 328 p^6 q^5 r - 128 q^6 r +  164 p q^6 r - 687 p^2 q^6 r \\
 & + 1000 p^3 q^6 r + 88 p^4 q^6 r + 1344 q^7 r - 96 p q^7 r - 1680 p^2 q^7 r - 3456 q^8 r -  96 p^5 q^2 r^2 + 352 p^6 q^2 r^2 - 512 p^7 q^2 r^2 +    352 p^8 q^2 r^2 - 96 p^9 q^2 r^2\\
 & + 320 p^2 q^3 r^2 - 912 p^3 q^3 r^2 + 768 p^4 q^3 r^2 + 816 p^5 q^3 r^2 -    1984 p^6 q^3 r^2 + 1184 p^7 q^3 r^2 - 768 q^4 r^2 +  1248 p q^4 r^2 - 738 p^2 q^4 r^2- 1116 p^3 q^4 r^2 \\
 & + 2886 p^4 q^4 r^2 - 3792 p^5 q^4 r^2 + 7452 q^5 r^2 -    4104 p q^5 r^2 - 16821 p^2 q^5 r^2 + 26082 p^3 q^5 r^2 -    18144 q^6 r^2 - 2592 p q^6 r^2 + 256 p^7 r^3 - 768 p^8 r^3  \\
 & + 768 p^9 r^3 - 256 p^10 r^3 - 96 p^3 q r^3 - 960 p^4 q r^3 + 
   3456 p^5 q r^3 - 4512 p^6 q r^3 + 2496 p^7 q r^3 - 384 p^8 q r^3 -  1728 p q^2 r^3 + 16848 p^2 q^2 r^3 \\
 &- 47736 p^3 q^2 r^3 +   70524 p^4 q^2 r^3 - 55296 p^5 q^2 r^3 + 20412 p^6 q^2 r^3 +  864 q^3 r^3 + 13068 p q^3 r^3 - 115830 p^2 q^3 r^3 + 211950 p^3 q^3 r^3  \\
 & - 141858 p^4 q^3 r^3 + 6480 q^4 r^3 -  30456 p q^4 r^3 + 114372 p^2 q^4 r^3 - 46656 q^5 r^3 +   6048 p^2 r^4 - 31104 p^3 r^4 + 67392 p^4 r^4 - 75168 p^5 r^4     \\
 & + 44928 p^6 r^4- 12096 p^7 r^4 - 5184 q r^4 + 14256 p q r^4 - 73872 p^2 q r^4 + 206064 p^3 q r^4 - 257904 p^4 q r^4 +111456 p^5 q r^4 + 49572 q^2 r^4   \\
 & - 49572 p q^2 r^4  +  43740 p^2 q^2 r^4 + 16038 p^3 q^2 r^4 - 131220 q^3 r^4 +  8748 p q^3 r^4 + 11664 r^5 - 29160 p r^5 - 61236 p^2 r^5 +  230364 p^3 r^5  \\
 & - 154548 p^4 r^5 - 17496 q r^5+ 34992 p q r^5 +    21870 p^2 q r^5 - 157464 q^2 r^5 - 78732 r^6 + 314928 p r^6)/(4p d_1 d_2 d_3),\\   
 &\tilde{r}_{53}= -p r (-2 p^2 q^7 + 6 p^3 q^7 - 6 p^4 q^7 + 2 p^5 q^7 + 8 q^8 -  21 p q^8 + 32 p^2 q^8 - 16 p^3 q^8 - 32 q^9 - 16 p q^9 - 16 p^2 q^5 r + 51 p^3 q^5 r - 46 p^4 q^5 r  \\
 & + 3 p^5 q^5 r  +  8 p^6 q^5 r + 64 q^6 r - 186 p q^6 r + 298 p^2 q^6 r - 
   270 p^3 q^6 r + 28 p^4 q^6 r - 560 q^7 r + 600 p q^7 r + 
   176 p^2 q^7 r + 1152 q^8 r - 32 p^2 q^3 r^2  \\
 & + 96 p^3 q^3 r^2- 48 p^4 q^3 r^2 - 200 p^5 q^3 r^2 + 320 p^6 q^3 r^2 - 
   136 p^7 q^3 r^2 + 128 q^4 r^2 - 384 p q^4 r^2 + 531 p^2 q^4 r^2 + 
   291 p^3 q^4 r^2 - 1434 p^4 q^4 r^2 \\
& + 1252 p^5 q^4 r^2- 1746 q^5 r^2 + 2100 p q^5 r^2 + 2790 p^2 q^5 r^2 - 
   6930 p^3 q^5 r^2 + 5040 q^6 r^2 + 360 p q^6 r^2 - 48 p^3 q r^3 +  368 p^4 q r^3 - 960 p^5 q r^3  \\
& + 1328 p^6 q r^3- 1008 p^7 q r^3   + 320 p^8 q r^3 + 96 p q^2 r^3 - 1536 p^2 q^2 r^3 +  5472 p^3 q^2 r^3 - 10968 p^4 q^2 r^3 + 11592 p^5 q^2 r^3 -  5328 p^6 q^2 r^3    \\
   & + 144 q^3 r^3 - 1854 p q^3 r^3  + 21105 p^2 q^3 r^3 - 44973 p^3 q^3 r^3 + 35226 p^4 q^3 r^3 -  3348 q^4 r^3 + 8424 p q^4 r^3 - 34560 p^2 q^4 r^3 +  15552 q^5 r^3   \\
 & - 432 p^2 r^4 + 2592 p^3 r^4 - 6336 p^4 r^4+ 8496 p^5 r^4 - 6624 p^6 r^4 + 2304 p^7 r^4 + 864 q r^4 -  2808 p q r^4 + 12312 p^2 q r^4 - 36936 p^3 q r^4   \\
 & +  52488 p^4 q r^4 - 26784 p^5 q r^4 - 12150 q^2 r^4 +13203 p q^2 r^4 - 13608 p^2 q^2 r^4 + 3888 p^3 q^2 r^4 + 
   37422 q^3 r^4 + 10692 p q^3 r^4 - 1944 r^5  \\
& + 3888 p r^5 + 19440 p^2 r^5 - 60264 p^3 r^5 + 42768 p^4 r^5+ 2916 q r^5 + 13122 p q r^5 - 51030 p^2 q r^5 + 52488 q^2 r^5 + 13122 r^6 \\
& -  52488 p r^6)  /(2q d_1 d_2 d_3),\\
&\tilde{r}_{54}=p r (-8 p^2 q^6 + 39 p^3 q^6 - 54 p^4 q^6 + 23 p^5 q^6 + 32 q^7 -   166 p q^7 + 246 p^2 q^7 - 182 p^3 q^7 - 128 q^8 + 536 p q^8 -  32 p^2 q^4 r + 216 p^3 q^4 r \\
& - 464 p^4 q^4 r + 408 p^5 q^4 r -   128 p^6 q^4 r + 128 q^5 r - 960 p q^5 r + 2069 p^2 q^5 r - 2065 p^3 q^5 r + 940 p^4 q^5 r + 72 q^6 r + 2586 p q^6 r -  3396 p^2 q^6 r \\
& - 2592 q^7 r + 304 p^3 q^2 r^2 - 1008 p^4 q^2 r^2 +  1184 p^5 q^2 r^2 - 560 p^6 q^2 r^2 + 80 p^7 q^2 r^2 -    1440 p q^3 r^2 + 4656 p^2 q^3 r^2 - 7056 p^3 q^3 r^2 \\
& +  5352 p^4 q^3 r^2  - 1608 p^5 q^3 r^2 + 2808 q^4 r^2 +  3564 p q^4 r^2 - 17658 p^2 q^4 r^2 + 17550 p^3 q^4 r^2 -  13392 q^5 r^2 - 1998 p q^5 r^2 - 192 p^4 r^3 + 192 p^5 r^3  \\
&  +  576 p^6 r^3 - 960 p^7 r^3 + 384 p^8 r^3 + 2160 p^2 q r^3 -  6048 p^3 q r^3 + 8640 p^4 q r^3 - 6480 p^5 q r^3 +    1728 p^6 q r^3 + 864 q^2 r^3 - 1296 p q^2 r^3 \\
&  -  37044 p^2 q^2 r^3   + 80892 p^3 q^2 r^3 - 58536 p^4 q^2 r^3 +  486 q^3 r^3 - 3483 p q^3 r^3 + 76626 p^2 q^3 r^3 - 34992 q^4 r^3 -   5184 p r^4 + 1296 p^2 r^4  \\
& + 38880 p^3 r^4 - 75168 p^4 r^4 + 40176 p^5 r^4 + 17496 q r^4 - 11664 p q r^4 + 52488 p^2 q r^4 -  58320 p^3 q r^4 - 84564 q^2 r^4 + 34263 p q^2 r^4 \\
& + 8748 p r^5  + 17496 p^2 r^5   - 118098 q r^5)/(4 d_1 d_2 d_3),\\
&\tilde{r}_{55}=r (-p^4 q^6 + 3 p^5 q^6 - 3 p^6 q^6 + p^7 q^6 + 8 p^2 q^7 -  24 p^3 q^7 + 34 p^4 q^7 - 18 p^5 q^7 + 8 q^8 - 48 p q^8 + 16 p^2 q^8 - 32 q^9 + 224 p q^9 - 16 p^4 q^4 r  \\
&+ 72 p^5 q^4 r-  120 p^6 q^4 r + 88 p^7 q^4 r - 24 p^8 q^4 r + 128 p^2 q^5 r - 
 592 p^3 q^5 r + 1164 p^4 q^5 r - 1144 p^5 q^5 r + 444 p^6 q^5 r - 64 q^6 r + 106 p q^6 r  \\
& - 792 p^2 q^6 r  + 1908 p^3 q^6 r - 1188 p^4 q^6 r + 672 q^7 r - 144 p q^7 r - 576 p^2 q^7 r - 
   1728 q^8 r - 48 p^4 q^2 r^2 + 224 p^5 q^2 r^2 - 432 p^6 q^2 r^2    \\
 & + 432 p^7 q^2 r^2 - 224 p^8 q^2 r^2+ 48 p^9 q^2 r^2 +  384 p^2 q^3 r^2 - 1952 p^3 q^3 r^2 + 4688 p^4 q^3 r^2 -  6288 p^5 q^3 r^2 + 4560 p^6 q^3 r^2 - 1392 p^7 q^3 r^2  \\
 &  -  384 q^4 r^2   + 720 p q^4 r^2  - 1665 p^2 q^4 r^2 + 4422 p^3 q^4 r^2 - 6012 p^4 q^4 r^2 + 2007 p^5 q^4 r^2 + 
   3726 q^5 r^2 - 2808 p q^5 r^2 - 7560 p^2 q^5 r^2   \\
& + 13716 p^3 q^5 r^2 - 9072 q^6 r^2 + 128 p^6 r^3 - 512 p^7 r^3+  768 p^8 r^3 - 512 p^9 r^3 + 128 p^10 r^3 - 384 p^3 q r^3 + 1152 p^4 q r^3 - 1152 p^5 q r^3   \\
& - 192 p^6 q r^3 + 1152 p^7 q r^3 -    576 p^8 q r^3 - 864 p q^2 r^3 + 9288 p^2 q^2 r^3 - 27432 p^3 q^2 r^3 + 43200 p^4 q^2 r^3 - 39312 p^5 q^2 r^3 +  17712 p^6 q^2 r^3  \\
 & + 432 q^3 r^3 + 5400 p q^3 r^3 -    62856 p^2 q^3 r^3 + 128952 p^3 q^3 r^3  - 96552 p^4 q^3 r^3 + 3240 q^4 r^3 - 10692 p q^4 r^3 + 62208 p^2 q^4 r^3 -  23328 q^5 r^3    \\
&  + 3888 p^2 r^4 - 21600 p^3 r^4 + 51408 p^4 r^4 - 67392 p^5 r^4 + 50544 p^6 r^4  - 16848 p^7 r^4 - 2592 q r^4 +  7776 p q r^4 - 42768 p^2 q r^4 + 137376 p^3 q r^4   \\
& -  194400 p^4 q r^4 + 97200 p^5 q r^4 + 24786 q^2 r^4 - 29889 p q^2 r^4 + 26973 p^2 q^2 r^4 - 12393 p^3 q^2 r^4 -  65610 q^3 r^4 + 21870 p q^3 r^4 + 5832 r^5   \\
&  - 23328 p r^5 -  17496 p^2 r^5 + 122472 p^3 r^5 - 87480 p^4 r^5 - 8748 q r^5 + 52488 p q r^5 - 26244 p^2 q r^5 - 78732 q^2 r^5 - 39366 r^6\\
&  +    216513 p r^6)/(2p d_1 d_2 d_3).    
\end{align*}
}

\section*{Acknowledgment}
The author would like to express his special gratitude to  Professor Hironori Shiga for valuable discussions about hypergeometric functions and $K3$ surfaces.
This work is supported by JSPS Grant-in-Aid for Scientific Research (22K03226)
and JST FOREST Program (JPMJFR2235).

\vspace{5mm}
\begin{center}
\hspace{10cm}\textit{Atsuhira  Nagano}\\
\hspace{10cm}\textit{Faculty of Mathematics and Physics}\\
\hspace{10cm} \textit{Institute of Science and Engineering}\\
\hspace{10cm}\textit{Kanazawa University}\\
\hspace{10cm}\textit{Kakuma, Kanazawa, Ishikawa}\\
\hspace{10cm}\textit{920-1192, Japan}\\
\hspace{10cm}\textit{(E-mail: atsuhira.nagano@gmail.com)}
\end{center}

\end{document}